\documentclass{article}
\usepackage{amsfonts}
\usepackage{amsmath}
\usepackage{enumerate}
\usepackage{hyperref}
\hypersetup{
    colorlinks=true,
    linkcolor=blue,
    filecolor=magenta,      
    urlcolor=blue,
}

\usepackage{tikz}
\usetikzlibrary{matrix,arrows,positioning,decorations.markings,decorations.pathmorphing}

\newtheorem{theorem}{Theorem}

\newtheorem{corollary}[theorem]{Corollary}

\newtheorem{definition}[theorem]{Definition}

\newtheorem{proposition}[theorem]{Proposition}
\newtheorem{remark}[theorem]{Remark}

\newenvironment{proof}[1][Proof]{\noindent\textbf{#1.} }{\ \rule{0.5em}{0.5em}}

\newcommand{\R}{\mathbb{R}}
\newcommand{\g}{\tilde{\mathfrak{g}}}
\newcommand{\sS}{\mathbb{S}}
\newcommand{\so}{\mathfrak{so}}

\begin{document}

\title{Lagrangian Reduction by Stages in Field Theory}
\author{Miguel \'{A}. Berbel \and Marco Castrill\'{o}n L\'{o}pez}
\date{}
\maketitle

\begin{abstract}
We propose a category of bundles in order to perform Lagrangian reduction by stages in covariant Field Theory. This category plays an analogous role to Lagrange-Poincar\'e bundles in Lagrangian reduction by stages in Mechanics and includes both jet bundles and reduced covariant configuration spaces.  Furthermore, we analyze the resulting reconstruction condition and formulate the Noether theorem in this context. Finally, a model of a molecular strand with rotors is seen as an application of this theoretical frame.
\end{abstract}


\medskip


\section{Introduction}
Symmetry represents the core of many (and probably the most important) tools developed to tackle dynamical systems. In particular, since the geometric formulation of Mechanics  (for example, Arnold \cite{arnold}, Marsden \cite{abraham}, Moser \cite{MS} and the references therein among other), special attention has been always given to those systems endowed with a group of symmetries. When the system is modelled on a manifold as the configuration space, the symmetry is expressed in terms of smooth actions of Lie groups. One natural  procedure is thus the construction of the quotient of the manifold by the group in the case of actions satisfying certain good properties. This is the so-called reduction procedure, which can be performed both in the Lagrangian and in the Hamiltonian geometric pictures of systems and that has attired and still attires the attention of many papers and books (a good reference of the, can be \cite{MRbook}). With the word \emph{system} we can include the evolution of particles governed by variational principles, symplectic forms or Poisson brackets, or time evolution of sections of bundles (of which Mechanics is a particular case) in a variational or multisymplectic approach. In full generality, systems describing section of bundles with no-prescribed time evolution (for example, covariant fields on space-times or geometric theories as harmonic maps) are also included in this versatile panorama.

Restricting ourselves to the Lagrangian or variational case, the Lagrangian functions are defined in the phase space of the system, a manifold including independent variables (as positions) together with their derivatives. The tangent bundle is the paradigmatic example in the case of Mechanics, which is generalized to jet spaces for Field Theories \cite{gotay}. The constructions of the variations, the variational principle and the equations for critical solutions is perfectly described in terms of geometric objects. In parallel, when reduction is performed, the new variations, the new variational principle as well as the new equations are now written in the reduced phase spaces that is not a tangent or jet space. This is the so-called Lagrange-Poincar\'e reduction first introduced in Mechanics (see \cite{CMR} for a historical account) and generalized for arbitrary bundles in \cite{cov.red} and \cite{LP}. An important consequence of the new nature of the reduced phase space arises when one needs to concatenate consecutive reductions. Indeed, there are many situations where the symmetry group is split in two or more parts of completely different properties. This difference may require a separate reduction for each part, a procedure called {\bf reduction by stages}. The work of Cendra, Marsden and Ratiu \cite{CMR} gave the convenient setting for this recursive reductions scheme that has been extensively used in the literature (just to mention some, the reader can go to \cite{Bl}, \cite{CF}, \cite{Fer}, \cite{Gay}, \cite{GS}). For that, a new category of phase spaces is introduced: the Lagrange-Poincar\'e category in Mechanics. See also \cite{BC} for the complete description of this category that closes some of the issues left open in \cite{CMR}.

The goal of this work is the construction of the Lagrange-Poincar\'e category for Field Theories. This includes the definition of the new phase manifolds, the variations, the variational principle and the equations for the critical sections of the configuration bundles. As we already learned from Lagrange-Poincar\'e reduction, these equations are split into two groups known as horizontal and vertical equations. Everything fits in a reduction program so that the reduced objects and principles by the action of a groups remain in the category and hence they can be object of repetitive reductions. This new category includes the Lagrange-Poincar\'e category in Mechanics as a particular case. The structure of the paper is as follows. In section 2 we provide the required preliminaries. This section also recalls the Lagrange-Poncar\'e category in Mechanics. In section 3, the Lagrange-Poincar\'e reduction principle for Field Theories is reviewed. In sections 4 and 5, the construction of the Lagrange-Poincar\'e category for Field Theories is introduced together with the detailed description of the variational principle and equations. As we mentioned before, all this follows the spirit and generalizes the particular case of Mechanics but, interestingly, this particularization can be pushed back and some properties of the general case can be directly derived from Mechanic as we show below. Section 6 analyzes and confirms the correct behavior of the theory when successive reductions are performed. Section 7 studies the reconstruction process from solution of the variational problem after reduction, to solutions of the unreduced problem. One gets the characteristic trait in Field Theories that one does not find in the Lagrange-Poincar\'e category in Mechanics: a compatibility condition (that is always satisfied in local neighborhoods) is needed to perform the reconstructions. This is already present since the first works on reduction for Field theories (see \cite{marco}, \cite{cov.red}) and continues in more recent works (see \cite{EPsubgroup}, \cite{LP}, \cite{CR}). As usual in these papers we describe it as the vanishing of the curvature of certain connection. If one thing can be intrinsically attached with the notion of symmetry is that of Noether current. In section 8 we explore this object in the new category and analyze its conservation. In fact, it is proved that the Noether current is not conserved in general, but it satisfies a specific drift law. The interesting property of this law is that is makes part of the vertical equation when reduction is performed. Roughly speaking, reductions make the vertical equation involve more and more variables (and hence, the horizontal equations become smaller and smaller) by adjoining the successive Noether drift laws to them. 

We complete the work in section 9 with an example. One paradigmatic instance of reduction by stages in Mechanics is the rigid body with rotors. Here we analyze the geometric setting describing a molecular strand composed by a chain of rigid bodies (as it is done in \cite{Mol Strand}) such that each body have one or more rotors. This could be regarded as a model of linked molecules with a rotating side chain(s). Simple proteins of aminoacids could fit in this context. Future work will include further applications of the theory to other models which can be inspired by the numerous applications of the reduction by stages in Mechanics or can be taken from new systems of purely covariant nature.

\section{Preliminaries}\label{preliminaries}
\subsection{Principal bundles.} Let $G$ be a Lie group acting freely and properly on the left on a manifold $Q$. Then, the quotient $Q/G$ is also a manifold and the projection $\pi_{Q/G,Q} :Q\rightarrow Q/G$ is a principal $G$-bundle. We recall that a \textit{principal connection }$\mathcal{A}$ on $Q\rightarrow Q/G$ is a $1$-form on $Q$ taking values on $\mathfrak{g}$, the Lie algebra of $G$, such that $\mathcal{A}(\xi _{q}^{Q})=\xi $, for any $\xi \in\mathfrak{g}$, $q\in Q$, and $\rho_{g}^*\mathcal{A}=\mathrm{Ad}_{g}\circ \mathcal{A}$, where $\rho_{g}:Q\to Q$ denotes the action by $g\in G$ and 
$$\xi ^{Q}_{q}=\left. \frac{d}{dt}\right\vert _{t=0}\exp (t\xi )\cdot q\in T_qQ.$$
Such a principal connection splits the tangent space as $T_{q}Q=H_{q}Q\oplus V_{q}Q$, for all $q\in Q$, where
$$V_{q}Q=\ker T_q\pi_{Q/G,Q}=\{v\in T_{q}Q|T_q\pi_{Q/G,Q}(v)=0\},\qquad q\in Q,$$
$$H_{q}Q=\ker \mathcal{A}_{q}=\{v\in T_{q}Q|\mathcal{A}_{q}(v)=0\},\qquad q\in Q,$$
are respectively called \textit{vertical} and \textit{horizontal subspace}. In fact, $H_{q}Q$ is isomorphic to $T_x(Q/G)$, $x=\pi_{Q/G,Q}(q)$, through $T_q\pi_{Q/G,Q}$. The inverse of this isomorphisms is called \emph{horizontal lift} and is denoted by $\mathrm{Hor}^{\mathcal{A}}_q$. The \textit{curvature} of a connection $\mathcal{A}$ is the $\mathfrak{g}$-valued $2$-form $$B(v,w)=d\mathcal{A}(\mathrm{Hor}(v),\mathrm{Hor}(w)),$$ where $v,w\in T_qQ$ and $\mathrm{Hor}(v)$ is the projection of $v\in T_qQ$ to $H_{q}Q$. 

The \textit{adjoint bundle} to $Q\to Q/G$ is the associated bundle $(Q\times \mathfrak{g})/G$ by the adjoint action of $G$ on $\mathfrak{g}$. We shall denote it by  $\mathrm{Ad} Q\rightarrow Q/G$ and its elements  by $[q,\xi]_{G}$, $q\in Q$, $\xi \in \mathfrak{g}$. Remarkably, $\mathrm{Ad} Q\to Q/G$ is a Lie algebra bundle since it is a vector bundle equipped with a fiberwise Lie bracket given by
$$\lbrack \lbrack q,\xi _{1}]_{G},[q,\xi _{2}]_{G}]=[q,[\xi _{1},\xi
_{2}]]_{G},\qquad \lbrack q,\xi _{1}]_{G},[q,\xi _{2}]_{G}\in
\mathrm{Ad} Q_{x},\hspace{2mm}x=\pi_{Q/G,Q} (q).$$
The principal connection $\mathcal{A}$ on $Q\to Q/G$ defines a linear connection on the vector bundle $\mathrm{Ad} Q\rightarrow Q/G$, denoted $\nabla^{\mathcal{A}}$ and given by the covariant derivative along curves
$$\frac{D[q(t),\xi(t)]_G}{Dt}=\left[q(t),\dot{\xi}(t)-[\mathcal{A}(\dot{q}(t)),\xi(t)] \right]_G.$$
In addition, the curvature of $\mathcal{A}$ can be seen as a 2-form on $Q/G$ with values in the adjoint bundle as for any $X,Y\in T_x(Q/G)$
$$\tilde{B}(X,Y)=[q,B(\mathrm{Hor}^{\mathcal{A}}_qX,\mathrm{Hor}^{\mathcal{A}}_qY)]_G.$$ 

The connection $\mathcal{A}$ induces a well-known vector bundle isomorphism
\begin{eqnarray}
\alpha_{\mathcal{A}}:TQ/G &\longrightarrow  &T(Q/G)\oplus \mathrm{Ad}Q \label{identification}\\
\left[v_q \right]_G  & \mapsto & T_q\pi_{Q/G,Q}(v_q)\oplus [q,\mathcal{A}(v_q)]_G  \notag
\end{eqnarray}
used in Mechanics to reduce $G$-invariant Lagrangians defined on $TQ$. This kind of reduction is called Lagrange--Poincar\'e reduction and it is described below.

\subsection{Quotient of vector bundles.} \label{quotientvector} Given a vector bundle $V\to Q$, we say that an action $\rho$ of a Lie group $G$ on $V$ is a vector bundle action if for all $g\in G$, $\rho_g:V\to V$ are  vector bundle isomorphisms and the action induced on $Q$ is free and proper. Then, there is a vector bundle structure on $V/G\to Q/G$ with operations
 $$[v_q]_G+[w_q]_G=[v_q+w_q]_G \text{ and } \lambda [v_q]_G=[\lambda v_q]_G,$$ where $[v_q]_G,[w_q]_G\in V/G$ stand for the equivalence classes of $v_q,w_q\in V_q$ and $\lambda \in \mathbb{R}$. In the  diagram
\begin{center}
\begin{tikzpicture}[scale=1.5] \label{quo.vect.ble}
\node (0) at (0,1) {$V$};
\node (A) at (1.5,1) {$V/G$};
\node (B) at (1.5,0) {$Q/G$};
\node (C) at (0,0) {$Q$};
\draw[->,font=\scriptsize,>=angle 90]
(0) edge node[above]{$\pi_{V/G,V}$} (A)
(0) edge node[right]{$\pi_{Q,V}$} (C)
(C) edge node[above]{$\pi_{Q/G,Q}$} (B)
(A) edge node[right]{$\pi_{Q/G,V/G}$} (B);
\end{tikzpicture}
\end{center}
 $\pi_{Q,V}$ and $\pi_{Q/G,V/G}$ are projections of vector bundles, $\pi_{Q/G,Q}$ is the projection of a $G$-principal bundle and  $\pi_{V/G,V}$ is a surjective vector bundle homomorphism.
 
Suppose that $V\to Q$ has an affine connection $\nabla$ or, equivalently, a covariant derivative $Dv(t)/Dt$ of curves $v(t)$ in $V$. Let $q(t)$ be a curve in $Q$ and denote $q_0=q(t_0)$ a fixed value of the curve, the \textit{horizontal lift} of $q(t)$ through $v\in V_{q_0}$ at $t_0$ is an horizontal curve in $V$ (that is, its covariant derivative vanishes) denoted $q^h_v(t)$ such that $\pi_{Q,V}\circ q^h_v=q$ and $q^h_v(t_0)=v$. 

Let $v(t)$ be a curve in $V$, $q(t)=\pi_{Q,V}(v(t))$, $q_0=q(t_0)$ and $x(t)=\pi_{Q/G,Q}(q(t))$. If $q_h(t)$ is the horizontal lift in $Q$ of $x(t)$ with respect to $\mathcal{A}$, we define the $g_{q_0}(t)$ in $G$ such that
$$q(t)=g_{q_0}(t)q_h(t).$$
Consider the curve $v_h(t)=g_{q_0}^{-1}(t)v(t)$ in $V$, then 
\begin{align*}
\frac{D}{Dt}\biggr\rvert_{t=t_0}v(t)=\frac{D}{Dt}\biggr\rvert_{t=t_0}g_{q_0}(t)v_h(t)=\frac{D}{Dt}\biggr\rvert_{t=t_0}g_{q_0}(t)v_h(t_0)+\frac{D}{Dt}\biggr\rvert_{t=t_0}g_{q_0}(t_0)v_h(t)\\
=\frac{D}{Dt}\biggr\rvert_{t=t_0}g_{q_0}(t)v(t_0)+\frac{D}{Dt}\biggr\rvert_{t=t_0}v_h(t)=\dot{g}_{q_0}(t_0)_{v(t_0)}^V+\frac{D}{Dt}\biggr\rvert_{t=t_0}v_h(t)
\end{align*}
Thus, the covariant derivative $Dv(t)/Dt$  at $t_0$ can be decomposed in horizontal and vertical components
$$\frac{D^{(\mathcal{A},H)}}{Dt}\biggr\rvert_{t=t_0}v(t)=\frac{D}{Dt}\biggr\rvert_{t=t_0}v_h(t), \hspace{10mm} \frac{D^{(\mathcal{A},V)}}{Dt}\biggr\rvert_{t=t_0}v(t)=\dot{g}_{q_0}(t_0)_{v(t_0)}^V.$$
Consequently, given $X\in\mathfrak{X}(Q)$ and $v\in\Gamma(Q,V)$ we can define
\begin{equation}
\label{2covariant}
\nabla_X^{(\mathcal{A},H)}v(q_0)=\frac{D^{(\mathcal{A},H)}}{Dt}\biggr\rvert_{t=t_0}v(t), \hspace{10mm} \nabla_X^{(\mathcal{A},V)}v(q_0)=\frac{D^{(\mathcal{A},V)}}{Dt}\biggr\rvert_{t=t_0}v(t),
\end{equation}
where $v(t)=v(q(t))$ and $q(t)$ is an integral curve of $X$ in $Q$ such that $q_0=q(t_0)$.

Let $X=Y\oplus \bar{\xi} \in \mathfrak{X}(Q/G)\oplus\Gamma(\mathrm{Ad}Q)\simeq \Gamma (TQ/G)$, using the identification \eqref{identification}. There is a unique $G$-invariant vector field $\bar{X}\in\Gamma^G(TQ)$  on $Q$ projecting to $X$. Furthermore, $\bar{X}=Y^h\oplus W$ with $Y^h \in\mathfrak{X}(TQ)$ the horizontal lift of $Y$ and $W$ the unique vertical $G$-invariant vector field such that for all $x\in Q/G$, $\bar{\xi}(x)=[q,\mathcal{A}(W(q))]_G$ with $q\in \pi_{Q/G,Q}^{-1}(x)$. Then, for $[v]_G\in \Gamma(Q/G,V/G)$ with $v\in\Gamma^G(Q,V)$ a $G$-invariant section, we define the \textit{quotient connection} by
$$\left[ \nabla^{(\mathcal{A})}\right] _{G,Y\oplus \bar{\xi}}[v]_G=[\nabla_{\bar{X}}v]_G,$$
the \textit{horizontal quotient connection} is defined by
$$\left[ \nabla^{(\mathcal{A},H)}\right] _{G,Y\oplus \bar{\xi}}[v]_G=[\nabla_{Y^h}v]_G=[\nabla^{(\mathcal{A},H)}_{\bar{X}}v]_G$$
and the \textit{vertical quotient connection} is defined by
$$\left[ \nabla^{(\mathcal{A},V)}\right] _{G,Y\oplus \bar{\xi}}[v]_G=[\nabla_{W}v]_G=[\xi_v^V]_G,$$
where $\xi$ satisfies $\bar{\xi}=[\pi_{Q,V}(v),\xi]_G.$ Note that these are not connections in the usual sense as derivation is performed with respect to sections of $TQ/G$ instead of sections of $T(Q/G)$. Only the horizontal quotient connection can be thought as a usual connection  since it only depends on $T(Q/G)\subset TQ/G$.

\subsection{The $\mathfrak{LP}$ category.} In Lagrange--Poincar\'e reduction the original Lagrangian $L$ is defined on $TQ$, the tangent bundle of the configuration space $Q$. However, the reduced Lagrangian is defined on $TQ/G\cong T(Q/G)\oplus \mathrm{Ad}Q$ which needs not be a tangent bundle. To iterate Lagrange--Poincar\'e reduction, the category $\mathfrak{LP}$ of Lagrange--Poincar\'e bundles was introduced in \cite{CMR}. Of course, this category, includes $TQ/G$ and is stable under reduction. 

The objects of $\mathfrak{LP}$ are vector bundles $TQ\oplus V\to Q$ where $TQ\to Q$ is the tangent bundle of a manifold $Q$, and $V\to Q$ is a vector bundle with the following additional structure:
\begin{enumerate}
\item \label{LP1a} a Lie bracket $[,]$ in the fibers of $V$;
\item \label{LP1b} a $V$-valued 2-form $\omega$ on $Q$;
\item \label{LP1c} a linear connection $\nabla$ on $V$;
\item \label{LP1d} and the bilineal operator defined by
$$[X_1\oplus w_1,X_2\oplus w_2]=[X_1,X_2]\oplus(\nabla_{X_1}w_2-\nabla_{X_2}w_1-\omega(X_1,X_2)+[w_1,w_2]),$$
where $[X_1,X_2]$ denotes the Lie bracket of vector fields and $[w_1,w_2]$ denotes the Lie bracket in the fibers of $V$, is a Lie bracket on sections $X\oplus w\in \Gamma(TQ\oplus V)$.
\end{enumerate}
We shall not detail the morphisms of $\mathfrak{LP}$ in this paper, instead they can be found in \cite{CMR}. Given a Lagrangian $L:TQ\oplus V\to\R$ defined on al element of $\mathfrak{LP}$, a curve $\dot{q}(t)\oplus v(t):[t_0,t_1]\to TQ\oplus V$ is said to be critical if and only if 
$$0=\left.\frac{d}{d\varepsilon}\right|_{\varepsilon =0}\int_{t_0}^{t_1} L(\dot{q}_{\varepsilon}(t)\oplus v_{\varepsilon}(t)) dt,$$
where $\dot{q}_{\varepsilon}(t)\oplus v_{\varepsilon}(t)$ is a variation of $\dot{q}(t)\oplus v(t)$ such that $\delta\dot{q}$ is the lifted variation of a free variation $\delta q$ and
 $$\delta v=\frac{Dw}{dt}+[v,w]+\omega_q(\delta q,\dot{q}),$$
where $w(t)$ is a curve in $V$ with $w(t_0)=w(t_1)=0$ and $\pi_{Q,V}(w(t))=q(t)$. This variational principle for Lagrangians is equivalent to the Lagrange--Poincar\'e equations
\begin{equation}
\frac{\delta L}{\delta q}-\frac{D}{Dt}\frac{\delta L}{\delta \dot{q}}=\left\langle\frac{\delta L}{\delta v},\omega_q(\dot{q},\cdot)\right\rangle ,
\end{equation}
\begin{equation}
\mathrm{ad}^*_{v}\frac{\delta L}{\delta v}=\frac{D}{Dt}\frac{\delta L}{\delta v},
\end{equation}
where $\mathrm{ad}^*$ stands for the coadjoint action in $V^*\to Q$.

\subsection{Reduction of $\mathfrak{LP}$ bundles.} Reduction in the Lagrange-Poincar\'e category is as follows:
\begin{proposition} \label{quotientLP}\emph{\cite[\S 6.2]{CMR}}
Let $TQ\oplus V \to Q$ be an object of $\mathfrak{LP}$ with additional structure $[,]$, $\omega$ and $\nabla$. Let $\rho:G\times (TQ\oplus V)\to TQ\oplus V$ be a free and proper action in the category $\mathfrak{LP}$ (for all $g\in G$, $\rho_g$ is an isomorphism in $\mathfrak{LP}$) and $\mathcal{A}$ a principal connection on $Q\to Q/G$. Then, the vector bundle
$$T(Q/G)\oplus\mathrm{Ad} Q\oplus(V/G)$$
with additional structures $[,]^{\tilde{\mathfrak{g}}}$, $\omega^{\tilde{\mathfrak{g}}}$ and $\nabla^{\tilde{\mathfrak{g}}}$ in $\mathrm{Ad} Q\oplus(V/G)$ given by
\begin{align*}
\nabla^{\tilde{\mathfrak{g}}}_X(\bar{\xi}\oplus[v]_G)=&\nabla^{\mathcal{A}}_X\bar{\xi}\oplus\left([\nabla^{(\mathcal{A},H)}]_{G,X}[v]_G-[\omega]_G(X,\bar{\xi}) \right),\\
\omega^{\tilde{\mathfrak{g}}}(X_1,X_2)=&\tilde{B}(X_1,X_2)\oplus [\omega]_G(X_1,X_2), \\
\,[\bar{\xi}_1\oplus[v_1]_G,\bar{\xi}_2\oplus[v_2]_G]^{\tilde{\mathfrak{g}}} =& [\bar{\xi}_1,\bar{\xi}_2] \oplus \left([\nabla^{(\mathcal{A},V)}]_{G,\bar{\xi}_1} [v_2]_G\right. \\ &- \left. [\nabla^{(\mathcal{A},V)}]_{G,\bar{\xi}_2}[v_1]_G
-[\omega]_G(\bar{\xi}_1,\bar{\xi}_2)+[[v_1]_G,[v_2]_G]_G  \right)
\end{align*}
is an object of the $\mathfrak{LP}$ category called the reduced bundle with respect to the group $G$ and the connection $\mathcal{A}$.
\end{proposition}
Suppose that $L:TQ\oplus V\to\R$ is $G$-invariant. Taking into account the identification \eqref{identification}  we can drop $L$ to a reduced Lagrangian 
$$l:T(Q/G)\oplus\mathrm{Ad} Q\oplus(V/G)\to\R .$$
Denote by $\pi_G$ the projection of $TQ\oplus V \to (TQ\oplus V)/G$ and $\alpha^{TQ\oplus V}_\mathcal{A}$ the identification between $(TQ\oplus V)/G$ and $T(Q/G)\oplus\mathrm{Ad} Q\oplus(V/G)$. A curve $\dot{q}(t)\oplus v(t)$ is critical for the variational problem set by $L$ if and only if the curve 
$$\dot{x}(t)\oplus \bar{\xi}(t) \oplus [v]_G(t)=\alpha^{TQ\oplus V}_\mathcal{A}\circ \pi_G (\dot{q}(t)\oplus v(t)),$$ 
is critical for the variational problem set by $l$ (see \cite{BC}). Equivalently, $\dot{q}(t)\oplus v(t)$ solves the Lagrange-Poincar\'e equations given by $L$ in $TQ\oplus V$ if and only if $\dot{x}(t)\oplus \bar{\xi}(t) \oplus [v]_G(t)$ solves the Lagrange--Poincar\'e equations given by $l$ in $T(Q/G)\oplus\mathrm{Ad} Q\oplus(V/G)$.
Hence, reduction can be made in this category and the procedure can be iterated: if we reduce by $N$, a normal subgroup of $G$ and afterward by $K = G/N$, the final result is equivalent to a direct reduction by $G$, whenever   the auxiliary connections used along the process are conveniently chosen.

\section{Lagrange--Poincar\'e field equations} \label{fiieldth}

Let $\pi_{X,P}:P\to X$ be a (non-necessarily principal) fiber bundle. Two local sections $\rho:U\to P$ and $\rho':U'\to P$ represent the same 1-jet, $j^1_x\rho$ at $x\in U\cap U'$ if and only if $\rho(x)=\rho'(x)$ and $T_x\rho=T_x\rho'$. This defines an equivalence relation and we denote by $J^1_xP$ the space of such classes. The \textit{1-jet bundle} is the space  $J^1P=\cup_{x\in X}J^1_xP$ equipped with a natural smooth structure of fiber bundle over $P$ with projection $j^1_x\rho\in J^1P\mapsto \rho(x)\in P$. The bundle $J^1P\to P$ is affine and modeled over the vector bundle $T^*X\otimes_PVP,$ where the abuse of notation $T^*X=\pi^*_{X,P}T^*X$ has been used.

A first order Lagrangian density is a smooth fiber map $\mathcal{L}:J^1P\to\bigwedge^n T^*X$, where $n=\mathrm{dim}(X)$. Suppose that $X$ is oriented and $\mathrm{Vol}\in\Gamma(\bigwedge^n T^*X)$ is a volume form, then the function $L: J^1P\to \R$ such that $\mathcal{L}=L\mathrm{Vol}$ is called a Lagrangian. A section $\rho$ of $P\to X$ is a critical section for the variatonal problem defined by $\mathcal{L}=L\mathrm{Vol}$ if
\begin{equation*}
\frac{d}{d\varepsilon}\biggr\rvert_{\varepsilon=0}\int_{X}\mathcal{L}(j^1\rho_{\varepsilon})=0
\end{equation*}
for all compactly supported variations $\rho_{\varepsilon}$ of $\rho$ that are vertical, that is for all $x\in X$, $d\rho_{\varepsilon}(x)/d\varepsilon\rvert_{\varepsilon=0}\in V_{\rho(x)}P$. This variational principle is equivalent to the fact that $\rho(x)$ satisfies the Euler--Lagrange equations, which can be written in an implicit way as
$$\frac{\delta L}{\delta\rho}-\mathrm{div}^{P}\frac{\delta L}{\delta j^1\rho}=0,$$
where $\delta L/\delta j^1\rho\in TX\otimes VP^*$ is the fiber derivative in $J^1P$, $\mathrm{div}^{P}$ is defined for $VP^*$-valued fields using a connection $\nabla^P$ in $(VP\subset TP)\to P$ and $\delta L/\delta\rho$ the horizontal differential with respect to $\nabla^P$.

Let $\Phi:G\times P\to P$ be a free and proper action such that for all $g\in G$, $\pi_{X,P}\circ\Phi_g=\pi_{X,P}$. Then $P\to\Sigma=P/G$ is a $G$-principal bundle and the action in $P$ can be lifted to $J^1P$. This defines $\pi_G:J^1P\to(J^1P)/G$ and, according to \cite{LP}, once fixed a connection $\mathcal{A}$ in $P\to\Sigma$, there is an identification
\begin{align}
\alpha_{\mathcal{A}}:(J^1P)/G\to &J^1\Sigma\oplus(T^*X\otimes_{\Sigma}\mathrm{Ad}P) \label{identification2}\\
j^1\rho\mapsto &j^1\sigma\oplus[\rho,\rho^*\mathcal{A}]_G \notag
\end{align}
where $\sigma=\pi_{\Sigma,P}\circ\rho = [\rho]_G$. We shall denote $\bar{\rho}=[\rho,\rho^*\mathcal{A}]_G$.Given a $G$-invariant Lagrangian $L:J^1P\to\R$ is a $G$-invariant Lagrangian, and its reduced Lagrangian $l:J^1\Sigma\oplus(T^*X\otimes_{\Sigma}\mathrm{Ad}P)\to\R$, the main result in \cite{LP} states that the variational principle defined by $L$ is equivalent to the fact that the reduced section $j^1\sigma\oplus\bar{\rho}(x)$ satisfies the Lagrange-Poincar\'e equations;
\begin{equation*}
\mathrm{ad}^*\frac{\delta l}{\delta\bar{\rho}}-\mathrm{div}^{\nabla}\frac{\delta l}{\delta\bar{\rho}}=0,
\end{equation*}
\begin{equation*}
\frac{\delta l}{\delta\sigma}-\mathrm{div}^{\Sigma}\frac{\delta l}{\delta j^1\sigma}=\langle\frac{\delta l}{\delta\bar{\rho}},i_{T_{\sigma}}\tilde{B}\rangle.
\end{equation*}
This is, in turn, equivalent to the variational principle
\begin{equation*}
\frac{d}{d\varepsilon}\biggr\rvert_{\varepsilon=0}\int_{X}l(j^1\sigma_{\varepsilon}\oplus\bar{\rho}_{\varepsilon})\mathrm{Vol}=0
\end{equation*}
for variations $j^1\sigma_{\varepsilon}\oplus\bar{\rho}_{\varepsilon}$ such that $\delta\bar{\rho}=\nabla^{\mathcal{A}}\bar{\eta}-[\bar{\eta},\bar{\rho}]-\tilde{B}(\delta\sigma,T\sigma)$, where $\delta\sigma$ is an arbitrary vertical variation of $\sigma$, $\bar{\eta}$ is an arbitrary section of $\mathrm{Ad}P\to X$ such that $\pi_{\Sigma, \mathrm{Ad}P}\bar{\eta}=\sigma$ and $\nabla^{\mathcal{A}}$ is the connection on $\mathrm{Ad}P$ induced by $\mathcal{A}$ and defined \S\ref{preliminaries}. This procedure is called Lagrange--Poincar\'e reduction for field theoretical covariant Lagrangians.

The attentive reader may have noticed that $\nabla^{\mathcal{A}}$ is a connection on $\mathrm{Ad}P\to \Sigma$ and, consequently, acts on sections of $\mathrm{Ad}P\to \Sigma$ while $\bar{\eta}$ is a section of $\mathrm{Ad}P\to X$. We shall now explain how to extend a connection to derive this kind of sections. Let $V\to P$ be a vector bundle with connection $\nabla$ and $P\to X$ a fiber bundle, given  $f:X\to V$ a section of $V\to X$, define the $\nabla$-derivative of $f$ with respect to $u_x\in T_xX$ as 
$$\tilde{\nabla}_{u_x}f=\frac{D^{\nabla}}{Dt}\biggr\rvert_{t=0}f(c(t))\in V_{\pi_{P,V}f(x)},$$
where $c(t)$ is a curve in $X$ such that $\dot{c}(0)=u_x$, $D^{\nabla}/Dt$ is the usual covariant derivative associated to $\nabla$ and $t\mapsto f(c(t))$ is a curve in $V$. As $f$ is a section on $V\to X$, $$\mathrm{id}_X=\pi_{X,V}\circ f=\pi_{X,P}\circ\pi_{P,V}\circ f$$ and $\rho=\pi_{P,V}\circ f$ is a section of $P\to X$. Then,
$f(c(t))=f\circ\pi_{X,P}\circ\rho(c(t))$ projects to curve $\rho(c(t))$ in $P$ with $D\rho(c(t))/dt\vert_{t=0}=T_x\rho(u_x)$ and
$$\tilde{\nabla}_{u_x}f=\frac{D^{\nabla}}{Dt}\biggr\rvert_{t=0}f(c(t))=\nabla_{T_x\rho(u_x)}\bar{f},$$ 
where $\bar{f}$ is a local section of $V\to P$ around $\rho(x)$ such that $\bar{f}\vert_{\mathrm{im}(\rho(X))}=f\circ\pi_{X,P}$. Sometimes we will use the abuse of notation $f=\bar{f}$.
\section{The FT$\mathfrak{LP}$ category}\label{FTLP}

We now define the category of bundles where reduction by stages for Field t
Theoretical covariant Lagrangians can be performed.

\begin{definition} Given $X$ a manifold called base space, the category $FT\mathfrak{LP}(X)$ of field theoretical Lagrange--Poincar\'e bundles over $X$ is defined as follows:
\begin{enumerate}[(A)]
\item The \textbf{objects} of  $FT\mathfrak{LP}(X)$ are bundles of the form $J^1P\oplus(T^*X\otimes_P V) \to P$, where $\pi_{XP}:P\to X$ is a bundle not necessarily principal, $T^*X\to P$ is an abuse of notation for the pullback $\pi^*_{XP}T^*X\to P$, and $V\to P$ is a vector bundle which is the vectorial part of an $\mathfrak{LP}$-bundle. In other words, $TP\oplus V \to P$ is an $\mathfrak{LP}$-bundle, which in turn is equivalent to the existence of
\begin{enumerate}[(i)]
\item a Lie bracket, $[,]$, in the fibers of $V$;
\item a $V$-valued 2-form $\omega$ on $P$;
\item a connection, $\nabla$, on $V\to P$;
\item a Lie bracket operation on the sections $Z\oplus u\in\Gamma(TP\oplus V)$ defined by 
$$[Z_1\oplus u_1,Z_2\oplus u_2]=[Z_1,Z_2]\oplus\nabla_{Z_1}u_2-\nabla_{Z_2}u_1-\omega(Z_1,Z_2)+[u_1,u_2]$$
\end{enumerate}
\item Let $J^1P_1\oplus(T^*X\otimes_{P_1} V_1) \to P_1$ and $J^1P_2\oplus(T^*X\otimes_{P_2} V_2) \to P_2$ be two  field theoretical Lagrange--Poincar\'e bundles over $X$ with structures $[,]_i$,$\nabla_i$ and $\omega_i$ on $V_i\to P_i$, $i=1,2$. A \textbf{morphism}, $f:J^1P_1\oplus(T^*X\otimes_{P_1} V_1) \to J^1P_2\oplus(T^*X\otimes_{P_2} V_2)$ is a bundle map covering $f_0:P_1\to P_2$ that satisfies
\begin{enumerate}[(i)]
\item $f_0:P_1\to P_2$ is a bundle map between $P_1\to X$ and $P_2\to X$ covering the identity on $X$,
\item $f(J^1P_1)\subset J^1P_2$ and $f_{\vert J^1P_1}=j^1f_0$, the 1-jet extension of $f_0$,
\item $f(T^*X\otimes_{P_1} V_1)\subset T^*X\otimes_{P_2} V_2$ and $f_{\vert T^*X\otimes_{P_1} V_1}=\mathrm{id}_{f_0}\otimes \bar{f}$, where 
\begin{align*}
\mathrm{id}_{f_0}\otimes \bar{f}:T^*X\otimes_{P_1} V_1 \to& T^*X\otimes_{P_2} V_2\\
(p_1,\alpha)\otimes v \mapsto & (f_0(p_1),\alpha)\otimes \bar{f}(v)
\end{align*}
$(p_1,\alpha)\in(\pi^*_{XP}T^*X)_{p_1}$ and $\bar{f}:V_1\to V_2$ is a vector bundle morphism covering $f_0$ and commuting with the structures on $V_i$ given by $[,]_i$,$\nabla_i$ and $\omega_i$ on $V_i\to P_i$, $i=1,2$. More explicitly, given $u, u' \in (V_1)_{p_1}$,
$Z, Z' \in (TP_1)_{p_1}$ and a curve $v(t)$ in $V_1$;
$$\bar{f}([u,u']_1)=[\bar{f}(u),\bar{f}(u')]_2,$$
$$\bar{f}(\omega_1(Z,Z'))=\omega_2(Tf_0(Z),Tf_0(Z')),$$
and
$$\bar{f}\left( \frac{D_1 v(t)}{Dt}\right)=\frac{D_2 \bar{f}(v(t))}{Dt},$$
are satisfied.
\end{enumerate}
\end{enumerate}
\end{definition} 
\begin{remark}
\em There are several special cases of objects in $FT\mathfrak{LP}(X)$ appearing in the present bibliography. For $V=0$, we obtain $1$-jet bundles used in Lagrangian covariant Field Theory (for example, see \cite{gotay}). Another instance of object in $FT\mathfrak{LP}(X)$ is the quotient of a $1$-jet bundle $J^1P$, by a proper and free lifted action of a Lie group $G$ on $P$ found to be isomorphic to $J^1(P/G)\oplus (T^*X\otimes_{(P/G)}\mathrm{Ad} P)$ in \cite{LP}. In the case where $P\to X$ is $G$-principal bundle, the quotient $J^1P/G$ is a $FT\mathfrak{LP}(X)$ bundle of the form $T^*X\otimes\mathrm{Ad} P$ which is the vector bundle underlying the affine bundle of connections used in \cite{marco} to perform Euler--Poincar\'e reduction.

Finally, the particular case when $X=\R$ and $P=\R\times Q$, with $Q$ a manifold, gives the $\mathfrak{LP}$ bundle
\[
T(\R\times Q)\oplus\mathrm{Ad}(\R\times Q)\simeq \mathbb{R} \times (TQ\oplus \mathrm{Ad}P)
\]
appearing when reducing time-dependent Lagrangians in classical Mechanics.   \em
\end{remark}

There exists a way of thinking bundles in $FT\mathfrak{LP}(X)$ as regular $\mathfrak{LP}$ bundles. First, we start by defining relevant subcategories of $\mathfrak{LP}$ bundles.

\begin{definition} Given $X$ a manifold called base space, we define the subcategory $\mathfrak{LP}(X)$ of $\mathfrak{LP}$ whose objects are $\mathfrak{LP}$ bundles $TP\oplus V$, such that $P$ is the total space of a fiber bundle $P \to X$, and whose morphisms are $\mathfrak{LP}$ morphisms, $f=T{f_0}\oplus\bar{f}: TP_1\oplus V_1\to TP_2\oplus V_2,$
such that $f_0:P_1\to P_2$ is a bundle map over $X$, that means, $f_0$ covers $\mathrm{id}_X$.
\end{definition}

\begin{proposition}\label{catequiv}
The applications
\begin{align*}
\mathcal{F}:\mathfrak{LP}(X)\to &FT\mathfrak{LP}(X)\\
TP\oplus V \mapsto &J^1P\oplus(T^*X\otimes V),
\end{align*}
\begin{align*}
\mathcal{F}:\mathrm{Hom}(\mathfrak{LP}(X))\to &\mathrm{Hom}(FT\mathfrak{LP}(X))\\
Tf_0\oplus\bar{f} \mapsto &j^1(f_0)\oplus(\mathrm{id}_{f_0}\otimes\bar{f}),
\end{align*}
define a covariant functor with inverse. Thence, the categories $\mathfrak{LP}(X)$ and $FT\mathfrak{LP}(X)$ are isomorphic.
\end{proposition}
\begin{proof}
Let $TP\oplus V\in\mathfrak{LP}(X)$, $$\mathcal{F}(\mathrm{id}_{TP\oplus V})=\mathcal{F}(T\mathrm{id}_P\oplus\mathrm{id}_V)=j^1(\mathrm{id}_P)\oplus(\mathrm{id}_{\mathrm{id}_P}\otimes\mathrm{id}_V)=\mathrm{id}_{\mathcal{F}(TP\oplus V)}.$$
On the other hand, given $f,g\in \mathrm{Hom}(\mathfrak{LP}(X))$ it is easy to see that $\mathcal{F}(g\circ f)=\mathcal{F}(g)\circ \mathcal{F}(f)$ as $T(g\circ f)=Tg\circ Tf$ and $j^1(g\circ f)=j^1g\circ j^1f$. These two properties prove that $\mathcal{F}$ is a functor. It has an inverse since 
\begin{align*}
\mathcal{G}:FT\mathfrak{LP}(X)\to &\mathfrak{LP}(X)\\
J^1P\oplus(T^*X\otimes V) \mapsto &TP\oplus V,
\end{align*}
\begin{align*}
\mathcal{G}:\mathrm{Hom}(FT\mathfrak{LP}(X))\to &\mathrm{Hom}(\mathfrak{LP}(X))\\
j^1(f_0)\oplus(\mathrm{id}_{f_0}\otimes\bar{f}) \mapsto &Tf_0\oplus\bar{f},
\end{align*}
is well defined, $\mathcal{G}\circ\mathcal{F}=\mathrm{id}_{\mathfrak{LP}(X)}$, and $\mathcal{F}\circ\mathcal{G}=\mathrm{id}_{FT\mathfrak{LP}(X)}$.
\end{proof}
\begin{corollary} \label{isos}
The following three statements are equivalent: $j^1(f_0)\oplus(\mathrm{id}_{f_0}\otimes\bar{f})$ is an isomorphism; $Tf_0\oplus\bar{f}$ is an isomorphism; and $\bar{f}$ is an isomorphism.
\end{corollary}  
\begin{proof}
The first two statements are equivalent since $\mathcal{F}$ is an isomorphism of categories. The third statement is equivalent to the others since $\bar{f}$ fully determines both $j^1(f_0)\oplus(\mathrm{id}_{f_0}\otimes\bar{f})$ and $Tf_0\oplus\bar{f}$.
\end{proof}

We shall define the notion of an action of a group $G$ on an object of $FT\mathfrak{LP}(X)$.
\begin{definition}
An action in the category $FT\mathfrak{LP}(X)$ of a group $G$ on an object $J^1P\oplus(T^*X\otimes V)$ of $FT\mathfrak{LP}(X)$ is $\Phi:G\times(J^1P\oplus(T^*X\otimes V))\to J^1P\oplus(T^*X\otimes V)$ such that for each $g\in G$, $\Phi_g:J^1P\oplus(T^*X\otimes V)\to J^1P\oplus(T^*X\otimes V)$ belongs to $\mathrm{Hom}(FT\mathfrak{LP}(X))$. We will say that this action is free and proper if the induced action on $P$ by the functions $(\Phi_g)_0$ is free and proper.
\end{definition}

\begin{proposition}
Let $J^1P\oplus(T^*X\otimes_P V)$ be an object of $FT\mathfrak{LP}(X)$ and let $[,]$, $\nabla$ and $\omega$ be the structure in $V$. Let $G$ be a Lie group acting freely and properly on  $J^1P\oplus(T^*X\otimes_P V)$ and $\mathcal{A}$ a connection in the principal bundle $P\to \Sigma=P/G$. The bundle $$J^1\Sigma\oplus(T^*X\otimes_{\Sigma} (\mathrm{Ad} P\oplus (V/G)))$$ with the structure $[,]^{\g}$, $\nabla^{\g}$ and $\omega^{\g}$ on $\mathrm{Ad} P\oplus V$ as in Proposition \ref{quotientLP} is an object in $FT\mathfrak{LP}(X)$ diffeomorphic to $(J^1P\oplus(T^*X\otimes_P V))/G$ via the bundle diffeomorphism
\begin{align}
\beta_{\mathcal{A}}:(J^1P\oplus(T^*X\otimes V))/G&\longrightarrow J^1\Sigma\oplus(T^*X\otimes (\textrm{ad} P\oplus (V/G))) \label{identificaction2}\\
[j^1_xs\oplus ((p,\alpha),v)]_G &\mapsto (T_{\pi_{\Sigma,P}}\circ j^1_xs, [p,s^*\mathcal{A}]_G+((\sigma,\alpha),[v]_G)),\notag
\end{align}
where $p\in P$, $\sigma=\pi_{\Sigma,P}(p)\in \Sigma$,  $x=\pi_{X,P}(p)\in X$ and $j^1_xs\in J^1_pP$.
\end{proposition}
\begin{proof}
As $G$ acts on $J^1P\oplus(T^*X\otimes V)$, for each $g\in G$ there exists an isomorphism $$\Phi_g=j^1((\Phi_g)_0)\oplus(\mathrm{id}_{(\Phi_g)_0}\otimes\bar{\Phi}_g):J^1P\oplus(T^*X\otimes V)\to J^1P\oplus(T^*X\otimes V)$$ of the category $FT\mathfrak{LP}(X)$.
Therefore, by Corollary \ref{isos}, $G$ acts on $TP\oplus V=\mathcal{F}^{-1}(J^1P\oplus(T^*X\otimes V))$ with isomorphisms 
$$\mathcal{F}^{-1}(\Phi_g)=T(\Phi_g)_0\oplus\bar{\Phi}_g:TP\oplus V\to TP\oplus V.$$
The quotient of $TP\oplus V$ by this action using connection $\mathcal{A}$ is isomorphic to another $\mathfrak{LP}$ bundle, $T\Sigma\oplus\mathrm{Ad} P\oplus (V/G)$ with structures $[,]^{\g}$, $\nabla^{\g}$ and $\omega^{\g}$. This bundle happens to be in $\mathfrak{LP}(X)$ since $\Sigma\to X$, then, from Proposition \ref{catequiv} $$\mathcal{F}(T\Sigma\oplus\mathrm{Ad} P\oplus (V/G))=J^1\Sigma\oplus(T^*X\otimes (\textrm{Ad} P\oplus (V/G)))$$ is isomorphic to $(J^1P\oplus(T^*X\otimes_P V))/G$. Finally, it can be checked that $\beta_{\mathcal{A}}$ is well defined and that
\begin{align*}
\beta^{-1}_{\mathcal{A}}:J^1\Sigma\oplus(T^*X\otimes (\textrm{ad} P\oplus (V/G)))\to &(J^1P\oplus(T^*X\otimes V))/G\\
\delta_{\sigma}\oplus\ell_{\sigma}\oplus ((\sigma,\alpha),[v]_G) \mapsto &[(\mathrm{Hor}^{\mathcal{A}}_p\circ\delta_{\sigma}+\kappa_p\circ\ell_{\sigma})\oplus((p,\alpha),v_p)]_G,
\end{align*}
where $v_p\in \pi^{-1}_{P,V}(p)=V_p$ such that $[v_p]_G=[v]_G$ and $\kappa_p:(\mathrm{Ad} P)_{\sigma} \to V_pP$ defined by $\kappa_p([p,\xi]_G)=\xi^P_p$.
\end{proof}

\section{Variational problems in FT$\mathfrak{LP}$ bundles}\label{FTLPmec}

Let $J^1P\oplus(T^*X\otimes_P V)$ be a  FT$\mathfrak{LP}$ bundle. A \textit{Lagrangian density} is a smooth fiber map $\mathcal{L}:J^1P\oplus(T^*X\otimes_P V)\to\bigwedge^{n}TX$ where $n$ the dimension of $X$. We will assume that $X$ is orientable and we choose a volume form $\mathrm{Vol}\in\bigwedge^{n}TX$, then, the Lagrangian density can be expressed as $\mathcal{L}=L\mathrm{Vol}$ with $L:J^1P\oplus(T^*X\otimes_P V)\to \R$.

Let $U\subset X$ be an open subset whose closure $\bar{U}$ is compact. We will only consider smooth sections $j^1\rho\oplus \nu: \bar{U}\to J^1P\oplus(T^*X\otimes_P V)$ such that $\nu\in\Gamma(\bar{U},T^*X\otimes_P V)$ projects to a section $\rho=\pi_{P,T^*X\otimes V}\circ\nu\in\Gamma(\bar{U},P)$, the $1$-jet extension of which is $j^1\rho$. These sections are called \textit{allowed sections}. We say that $j^1\rho\oplus \nu\in\Gamma(\bar{U}, J^1P\oplus(T^*X\otimes_P V))$ is a \textit{critical section} for the variational problem defined by $\mathcal{L}$ if 
\begin{equation}
\frac{d}{d\varepsilon}\biggr\rvert_{\varepsilon=0}\int_{U}\mathcal{L}(j^1\rho_{\varepsilon}\oplus \nu_{\varepsilon})=0
\end{equation}
for all smooth \textit{allowed variations} $j^1\rho_{\varepsilon}\oplus \nu_{\varepsilon}$. However, in order to define the set of allowed variations we will first introduce a connection in $T^*X\otimes_P V$ from a connection $\nabla^X$ in $TX\to X$ and a connection $\nabla$ on $V\to P$. In fact, given $\nu\in \Gamma(P,T^*X\otimes_P V)$, $Z\in\mathfrak{X}(P)$ and $u\in\mathfrak{X}(U)$, the connection $\nabla^L$ on $T^*X\otimes_P V\to P$ is given by
\begin{equation}\label{connlin}
(\nabla^{L}_Z\nu)(u)=\nabla_Z(\nu(u))-\nu(\nabla^X_{T\pi_{X,P}(Z)}u).
\end{equation}
 
\begin{definition}
An allowed variation of an allowed section, $j^1\rho\oplus \nu: \bar{U}\to J^1P\oplus(T^*X\otimes_P V)$,  is a smooth map $j^1\rho(x,\varepsilon)\oplus \nu(x,\varepsilon): \bar{U}\times I \to J^1P\oplus(T^*X\otimes_P V)$, where $I$ is an open interval with $0\in I$, such that:
\begin{enumerate}
\item for all $\varepsilon\in I$, $j^1\rho_{\varepsilon}(x)\oplus \nu_{\varepsilon}(x):\bar{U}\to J^1P\oplus(T^*X\otimes_P V)$ is an allowed section and $j^1\rho_{\varepsilon}\oplus \nu_{\varepsilon}\vert_{\partial U}=j^1\rho\oplus \nu\vert_{\partial U}$;
\item for $\varepsilon=0$, $j^1\rho_{\varepsilon}(x)\oplus \nu_{\varepsilon}(x)=j^1\rho(x)\oplus \nu(x)$;
\item the variation of $\nu$ is of the form $$\delta\nu\equiv\frac{D^{\nabla^L}\nu_{\varepsilon}}{D\varepsilon}\biggr\rvert_{\varepsilon=0}=\tilde{\nabla}\mu-[\mu,\nu]+\rho^*(i_{\delta\rho}\omega),$$
where $\omega$ is the $2$-form in the additional structure of $V$,$\tilde{\nabla}$ is the $\nabla$-derivative of $(V,\nabla)\to P\to X$ and $\mu\in\Gamma(\bar{U},V)$ an arbitrary  section with $\pi_{P,V}\circ\mu=\rho$ and  $\mu\vert_{\partial U}=0$.
\end{enumerate} 
\end{definition}
\begin{remark}
Let $j^1\rho_{\varepsilon}\oplus \nu_{\varepsilon}$ be an allowed variation of $j^1\rho\oplus \nu$, the variation of $\rho_{\varepsilon}=\pi_{P,T^*X\otimes_P V}\circ\nu_{\varepsilon}$ is vertical in the sense that 
$$\delta\rho(x)\equiv\frac{d\rho_{\varepsilon}(x)}{d\varepsilon}\biggr\rvert_{\varepsilon=0}\in V_{\rho(x)}P=\mathrm{ker}(T_{\rho(x)}\pi_{X,P}).$$
Consequently, $\delta j^1\rho(x)\in V_{\rho(x)}J^1P=\mathrm{ker}(T_{j^1\rho(x)}\pi_{X,J^1P})$
\end{remark}

\begin{remark}
It is very important to realize that, since $\rho=\pi_{P,T^*X\otimes_P V}\nu$, an allowed section $(j^1\rho,\nu)$ is completely defined by $\nu\in\Gamma(\bar{U},T^*X\otimes_P V)$. In a similar way, an allowed variation is completely determined by $\delta\nu$ since $\delta \rho$ is its projection along the map $V\to X$.
\end{remark}
\begin{remark}
\em Given $u_x\in T_xX$,
$$\delta\nu(u_x)\equiv\frac{D^{\nabla^L}\nu_{\varepsilon}}{D\varepsilon}\biggr\rvert_{\varepsilon=0}(u_x)=\frac{D^{\nabla}\nu_{\varepsilon}(u_x)}{D\varepsilon}\biggr\rvert_{\varepsilon=0}-\nu(\nabla^X_{T\pi_{X,P}(\delta\rho)}u)=\frac{D^{\nabla}\nu_{\varepsilon}(u_x)}{D\varepsilon}\biggr\rvert_{\varepsilon=0},$$
where $u$ is a local section of $TX\to X$ around $x$ such that $u(x)=u_x$ and the last equivalence is a consequence of $\delta\rho$ being vertical. Thus, the variation $\delta\nu$ does not depend on the connection $\nabla^X$ on $TX\to X$.\em
\end{remark}

We now find the variational equations defined by these set of allowed sections and variations. For any smooth function $L:J^1P\oplus(T^*X\otimes_P V)\to\R$, the fiber derivatives are defined as 
\begin{equation} \label{fibderj}
\left\langle\frac{\delta L}{\delta j^1\rho}(j^1\rho\oplus\nu),\alpha \right\rangle =\frac{d}{d\epsilon}\biggr\rvert_{\epsilon=0}L((j^1\rho+\epsilon\alpha)\oplus\nu);
\end{equation} 
\begin{equation} \label{fibderv}
\left\langle\frac{\delta L}{\delta \nu}(j^1\rho\oplus\nu),\beta \right\rangle =\frac{d}{d\epsilon}\biggr\rvert_{\epsilon=0}L(j^1\rho\oplus(\nu+\epsilon\beta)),
\end{equation}
where $\alpha\in T^*X\otimes VP$ and $\beta\in T^*X\otimes V$. Therefore, $$\frac{\delta L}{\delta j^1\rho}(j^1\rho\oplus\nu)\in TX\otimes VP^*,\hspace{10mm}\frac{\delta L}{\delta \nu}(j^1\rho\oplus\nu)\in TX\otimes V^*,$$ and if we compose them with a section $j^1\rho(x)\oplus\nu(x)\in\Gamma(\bar{U},J^1P\oplus(T^*X\otimes V))$ we obtain $\delta L/ \delta j^1\rho(x)$ in $\Gamma(\bar{U}, TX\otimes VP^*)$ and $\delta L/ \delta\nu(x)$ in $\Gamma(\bar{U}, TX\otimes V^*)$. 

Given $\nabla^P$ a linear connection on $TP\to P$ and $\nabla^X$ a linear connection in $TX\to X$, using the dual and the product connections, there is a lineal connection in $T^*X\otimes_PTP$. According to \cite{Janyska}, this linear connection induces a general connection $\nabla^{J^1P}$ on $J^1P\to P$ that does not depend on the choice of $\nabla^X$. Furthermore, we can choose $\nabla^P$ to be projectable to $\nabla^X$ and the connection in $\nabla^{J^1P}$ will be affine. In addition, we define an affine connection, $\nabla^{J^1P}\oplus\nabla^L$, in $J^1P\oplus(T^*X\otimes_P V)$. Therefore, we define an horizontal derivative,
\begin{equation} \label{fibderh}
\left\langle\frac{\delta L}{\delta \rho}(j^1\rho\oplus\nu),Z_p \right\rangle =\frac{d}{d\epsilon}\biggr\rvert_{\epsilon=0}L(\zeta^h_{j^1\rho\oplus\nu}(\epsilon));
\end{equation} 
where $Z_p\in T_pP$ and $\zeta(\epsilon)$ is a curve in $P$ such that $\dot{\zeta}(0)=Z_p$ and $\zeta^h_{j^1\rho\oplus\nu}(\epsilon)$ is the horizontal lift of $\zeta(\epsilon)$ to $(J^1P\oplus(T^*X\otimes V))$ through $j^1\rho\oplus\nu$ using the connection $\nabla^{J^1P}\oplus\nabla^L$. Thus, $$\frac{\delta L}{\delta \rho}(j^1\rho\oplus\nu)\in T^*P.$$

We will also need a general notion of divergence of fields with values in a vector bundle. Let $E\to P$ be a vector bundle with affine connection $\nabla$ and $P\to X$ a fiber bundle, we define for any $\chi\in\Gamma(X,TX\otimes E^*)$ the divergence $\mathrm{div}^{\nabla}\chi\in \Gamma(X,E^*)$ such that, for any $\eta\in\Gamma(X,E)$,
\begin{align*}
\mathrm{div}\langle\xi,\eta\rangle=\langle\mathrm{div}^{\nabla}\chi,\eta\rangle+\langle\chi,\tilde{\nabla}\eta\rangle,
\end{align*}
where $\mathrm{div}$ is the usual divergence of a vector field in $X$ (with respect to the volume form $\mathrm{Vol}$). In our exposition we will use the operators
$$\mathrm{div}^{\nabla}:\Gamma(\bar{U}, TX\otimes V^*)\to \Gamma(\bar{U}, V^*);$$
$$\mathrm{div}^P:\Gamma(\bar{U}, TX\otimes VP^*)\to \Gamma(\bar{U}, VP^*);$$
induced by the connection $\nabla$ in $V\to P$, and the restriction of the connection $\nabla ^P$ to $VP\subset TP$.

Finally, we define the coadjoint operator in this context as
\begin{align*}
\mathrm{ad}^*:\Gamma(\bar{U},T^*X\otimes V)\times\Gamma(\bar{U},TX\otimes V^*)\to&\Gamma(\bar{U},V^*)\\
(\nu_1,\nu_2)\mapsto&(\mu\mapsto\mathrm{ad}_{\nu_1}\nu_2(\mu)=\langle\nu_2,[\nu_1,\mu]\rangle
\end{align*}
for all $\mu\in\Gamma(\bar{U},V)$.

\begin{proposition}\label{LPeq}
Let $J^1P\oplus(T^*X\otimes V)$ be a FT$\mathfrak{LP}$ bundle with a Lagrangian density $\mathcal{L}:J^1P\oplus(T^*X\otimes V)\to\bigwedge^{n}TX$ and a volume form $\mathrm{Vol}\in\bigwedge^{n}TX$, such that $\mathcal{L}=L\mathrm{Vol}$. Let $\nabla ^P$ a linear connection in $TP\to P$. Then an allowed section $j^1\rho\oplus\nu\in \Gamma(\bar{U},J^1P\oplus(T^*X\otimes V))$ is critical for the variational problem defined by $\mathcal{L}$ if and only if it satisfies the Lagrange-Poincar\'e equations:
\begin{equation}
\mathrm{ad}^*\frac{\delta L}{\delta\nu}-\mathrm{div}^{\nabla}\frac{\delta L}{\delta\nu}=0,
\end{equation}
\begin{equation}
\frac{\delta L}{\delta\rho}-\mathrm{div}^{P}\frac{\delta L}{\delta j^1\rho}-\langle\frac{\delta L}{\delta j^1\rho},i_{T_{\rho}}T^P\rangle=\langle\frac{\delta L}{\delta\nu},i_{T_{\rho}}\omega\rangle,
\end{equation}
where $T^P$ is the torsion tensor of connection $\nabla^P$. Since this connection is arbitrary we can always choose a connection without torsion and remove this term.
\end{proposition}
\begin{proof}
Using the derivatives defined by equations \ref{fibderj}, \ref{fibderv} and \ref{fibderh}, we rewrite the derivative of the action as;
\begin{align*}
&\frac{d}{d\varepsilon}\biggr\rvert_{\varepsilon=0}\int_{U}\mathcal{L}(j^1\rho_{\varepsilon}\oplus \nu_{\varepsilon})=\frac{d}{d\varepsilon}\biggr\rvert_{\varepsilon=0}\int_{U}L(j^1\rho_{\varepsilon}\oplus \nu_{\varepsilon})\mathrm{Vol}=\\
&=\int_U \left\langle \frac{\delta L}{\delta\rho}(x),\frac{d\rho_{\varepsilon}(x)}{d\varepsilon}\biggr\rvert_{\varepsilon=0}\right\rangle\mathrm{Vol} + \int_U\left\langle\frac{\delta L}{\delta j^1\rho}(x), \frac{D^{\nabla^{J^1P}}j^1\rho_{\varepsilon}(x)}{D\varepsilon}\biggr\rvert_{\varepsilon=0}\right\rangle\mathrm{Vol} \\ &\phantom{al}+\int_U\left\langle \frac{\delta L}{\delta\nu}(x), \frac{D^L\nu_{\varepsilon}(x)}{D\varepsilon}\biggr\rvert_{\varepsilon=0}\right\rangle \mathrm{Vol}.
\end{align*}
We know from the definition of allowed variation that
\begin{equation} \label{variation}
\frac{D^{\nabla^L}\nu_{\varepsilon}}{D\varepsilon}\biggr\rvert_{\varepsilon=0}=\tilde{\nabla}\mu-[\mu,\nu]+\rho^*(i_{\delta\rho}\omega).
\end{equation} On the other hand, since the variation of $\rho$ is vertical and $\nabla^P$ is projectable,
\begin{align*}
\frac{D^{\nabla^{J^1P}}j^1\rho_{\varepsilon}(x)}{D\varepsilon}\biggr\rvert_{\varepsilon=0}(u_x)=\frac{D^{\nabla^P}j^1\rho_{\varepsilon}(x)(u_x)}{D\varepsilon}\biggr\rvert_{\varepsilon=0},
\end{align*}
for all $u_x\in T_xX$. Consider $\rho_{\varepsilon}(\gamma(t))$ where $\gamma(t)$ is a curve such that $\dot{\gamma}(0)=u_x$. From the formula
$$\frac{D^{\nabla^P}}{D\varepsilon}\frac{d}{dt}\rho_{\varepsilon}(\gamma(t))-\frac{D^{\nabla^P}}{Dt}\frac{d}{d\varepsilon}\rho_{\varepsilon}(\gamma(t))=T^P\left(\frac{d}{d\varepsilon}\rho_{\varepsilon}(\gamma(t)),\frac{d}{dt}\rho_{\varepsilon}(\gamma(t)) \right),$$
we get that 
\begin{equation} \label{varjp}
\frac{D^{\nabla^{J^1P}}j^1\rho_{\varepsilon}(x)}{D\varepsilon}\biggr\rvert_{\varepsilon=0}(u_x)=\tilde{\nabla}^P_{u_x}\delta\rho(x)+T^P(\delta\rho(x),T_x\rho(u_x))
\end{equation}
After substitution of equations \ref{variation} and \ref{varjp}, the derivative of the action is
\begin{align*}
&\frac{d}{d\varepsilon}\biggr\rvert_{\varepsilon=0}\int_{U}\mathcal{L}(j^1\rho_{\varepsilon}\oplus \nu_{\varepsilon})=\\
&=\int_U \left\langle \frac{\delta L}{\delta\rho}(x),\delta\rho(x)\right\rangle\mathrm{Vol} + \int_U\left\langle\frac{\delta L}{\delta j^1\rho}(x), \tilde{\nabla}^P\delta\rho(x)+T^P(\delta\rho(x),T\rho)
\right\rangle\mathrm{Vol} \\ &\phantom{al}+\int_U\left\langle \frac{\delta L}{\delta\nu}(x), \tilde{\nabla}\mu-[\mu,\nu]+\omega(\delta\rho(x),T\rho)\right\rangle \mathrm{Vol}\\
&=\int_U \left\langle \frac{\delta L}{\delta\rho}(x)-\mathrm{div}^P\frac{\delta L}{\delta j^1\rho}(x)-\left\langle \frac{\delta L}{\delta j^1\rho},i_{T\rho}T^P\right\rangle-\left\langle \frac{\delta L}{\delta\mu},i_{T\rho}\omega\right\rangle ,\delta\rho(x)\right\rangle\mathrm{Vol}\\
&\phantom{al}+\int_U\left\langle -\mathrm{div}^{\nabla}\frac{\delta L}{\delta\nu}(x)+\mathrm{ad}_{\nu},\mu\right\rangle \mathrm{Vol},
\end{align*}
where for the second identity it has been used that $\delta\rho\vert_{\partial U}=0$ and $\mu\vert_{\partial U}=0$. Finally, from the last expression is clear that $j^1\rho(x)\oplus\nu(x)$ is critical if and only if the Euler--Poincar\'e equations are satisfied.
\end{proof}

\section{Reduction by stages}
In this section we shall see that the reduction procedure can be performed in the category  FT$\mathfrak{LP}$. Let $J^1P\oplus (T^*X\otimes V)$ be an object in FT$\mathfrak{LP}$, $G$ a Lie group acting freely and properly on $J^1P\oplus (T^*X\otimes V)$ and $\mathcal{A}$ a connection on $P\to\Sigma=P/G$. We recall from  subsection \ref{quotientvector} that the connection $\nabla$ on $V\to P$ induces an affine connection $\left[ \nabla^{(\mathcal{A},H)}\right] _{G,Y}$ on $V/G\to\Sigma$, for $Y\in\mathfrak{X}(\Sigma)$. Hence, there is a $\left[ \nabla^{(\mathcal{A},H)}\right] _{G}$-derivative on $V/G\to\Sigma\to X$ denoted by $\left[ \tilde{\nabla}^{(\mathcal{A},H)}\right] _{G}$. On the other hand, the vertical component of the reduced connection,  $\left[ \nabla^{(\mathcal{A},V)}\right] _{G,\bar{\xi}}$, is not a connection since we derive with respect to $\bar{\xi}\in \Gamma(\mathrm{Ad}P)$. However, given a section $[w]$ of $V/G\to X$, we define for all $x\in X$
$$\left[ \nabla^{(\mathcal{A},V)}\right] _{G,\xi}[w](x)=[\xi^V_{v}]_G(x),$$
where $v\in V$ such that $\pi_{V/G,V}(v)=[w](x)$, and $\bar{\xi}=[p,\xi]_G$ with $p=\pi_{V,P}(v)$. This is well defined since $\bar{\xi}=[gp,\mathrm{Ad}_g\xi]_G$ and $g\xi^V_{v}=(\mathrm{Ad}_G\xi)^V_{gv}.$


\begin{proposition}\label{reduction}
Let $J^1P\oplus (T^*X\otimes V)$ be an object in FT$\mathfrak{LP}$, $G$ a Lie group acting freely and properly on $J^1P\oplus (T^*X\otimes V)$, $\mathrm{Vol}$ a volume form on $X$ and $L:J^1P\oplus (T^*X\otimes V)\to\R$ a $G$-invariant Lagrangian.

Consider $\mathcal{A}$ a connection on $P\to\Sigma=P/G$ and 
$$l:J^1\Sigma\oplus (T^*X\otimes (\mathrm{Ad}Q\oplus (V/G)))\to\R$$
the reduced Lagrangian induced in the quotient via the identification \eqref{identification2}. Given a smooth local section $j^1\rho\oplus\nu\in \Gamma(\bar{U},J^1P\oplus (T^*X\otimes V))$, where $\nu=((\rho,\alpha),v)$, we define the reduced local section
$$j^1\sigma\oplus\bar{\rho}\oplus[\nu]_G=\beta_{\mathcal{A}}\circ\pi_G(j^1\rho\oplus\nu)=(T\pi_{P,\Sigma}\circ j^1\rho)\oplus[\rho,\rho^*\mathcal{A}]_G\oplus((\sigma,\alpha),v),$$
where $\sigma\in\Gamma(\bar{U},\Sigma)$ is $\pi_{\Sigma,P}\circ\rho$. Then, the following statements are equivalent
\begin{enumerate}[(i)]
\item Section $j^1\rho\oplus\nu\in \Gamma(\bar{U},J^1P\oplus (T^*X\otimes V))$ is a critical section for the variational problem defined by $L$ on $J^1P\oplus (T^*X\otimes V)$.
\item Section $j^1\rho\oplus\nu\in \Gamma(\bar{U},J^1P\oplus (T^*X\otimes V))$ satisfies the Lagrange-Poincar\'e equations given by $L$ in $J^1P\oplus (T^*X\otimes V)$.
\item Section $j^1\sigma\oplus\bar{\rho}\oplus[\nu]_G\in \Gamma(\bar{U},J^1\Sigma\oplus (T^*X\otimes (\mathrm{Ad}Q\oplus (V/G))))$ is a critical section for the variational problem defined by $l$ on $J^1\Sigma\oplus (T^*X\otimes (\mathrm{Ad}Q\oplus (V/G)))$.
\item Section $j^1\sigma\oplus\bar{\rho}\oplus[\nu]_G\in \Gamma(\bar{U},J^1\Sigma\oplus (T^*X\otimes (\mathrm{Ad}Q\oplus (V/G))))$ satisfies the Lagrange-Poincar\'e equations given by $l$ in $J^1\Sigma\oplus (T^*X\otimes (\mathrm{Ad}Q\oplus (V/G)))$.
\end{enumerate}
\end{proposition}
\begin{proof}
From Proposition \ref{LPeq}, statements (i) and (ii) are equivalent. In an analogue way statements (iii) and (iv) are equivalent. We will prove the result by checking that statements (i) and (iii) are equivalent. 

Let $j^1\rho_{\varepsilon}\oplus\nu_{\varepsilon}$ be an allowed variation of $j^1\rho\oplus\nu$ and $$j^1\sigma_{\varepsilon}\oplus\bar{\rho}_{\varepsilon}\oplus[\nu]_{G,\varepsilon}=\beta_{\mathcal{A}}\circ\pi_G(j^1\rho_{\varepsilon}\oplus\nu_{\varepsilon}),$$
the projection of the allowed variation in $J^1P\oplus (T^*X\otimes V)$. Since 
\begin{align*}
&\frac{d}{d\varepsilon}\biggr\rvert_{\varepsilon=0}\int_{U}L(j^1\rho_{\varepsilon}\oplus \nu_{\varepsilon})\mathrm{Vol}=\frac{d}{d\varepsilon}\biggr\rvert_{\varepsilon=0}\int_{U}l(j^1\sigma_{\varepsilon}\oplus\bar{\rho}_{\varepsilon}\oplus[\nu]_{G,\varepsilon})\mathrm{Vol}=\\
&=\int_U \left\langle \frac{\delta l}{\delta\sigma}(x),\frac{d\sigma_{\varepsilon}(x)}{d\varepsilon}\biggr\rvert_{\varepsilon=0}\right\rangle\mathrm{Vol} + \int_U\left\langle\frac{\delta l}{\delta j^1\sigma}(x), \frac{D^{\nabla^{J^1\Sigma}}j^1\sigma_{\varepsilon}(x)}{D\varepsilon}\biggr\rvert_{\varepsilon=0}\right\rangle\mathrm{Vol} \\ &\phantom{al}+\int_U\left\langle \frac{\delta l}{\delta(\bar{\rho}\oplus[\nu]_G)}(x), \frac{D^L\bar{\rho}_{\varepsilon}\oplus[\nu]_{G,\varepsilon}(x)}{D\varepsilon}\biggr\rvert_{\varepsilon=0}\right\rangle \mathrm{Vol},
\end{align*}
the variational problem defined by $L$ on $J^1P\oplus (T^*X\otimes V)$ is equivalent to consider variations of $j^1\sigma\oplus \bar{\rho}\oplus[\nu]_G$ obtained by projecting allowed variations of $j^1\rho\oplus\nu$. Then, for all $u_x\in T_xX$,
\begin{align*}
&\frac{D^L\bar{\rho}_{\varepsilon}\oplus[\nu]_{G,\varepsilon}}{D\varepsilon}\biggr\rvert_{\varepsilon=0}(u_x)=\frac{D^{\g}\bar{\rho}_{\varepsilon}\oplus[\nu]_{G,\varepsilon}(u_x)}{D\varepsilon}\biggr\rvert_{\varepsilon=0}\\
&=\nabla_{\delta\sigma}^{\g}\bar{\rho}_{\varepsilon}(u_x)\oplus\left( [\nabla^{(\mathcal{A},H)}]_{G,\delta\sigma}[v]_G(u_x)-[\omega]_G(\delta\sigma,\bar{\rho}_{\varepsilon}(u_x))\right)\\
&=\nabla_{\delta\sigma}^{\g}\bar{\rho}_{\varepsilon}(u_x)\oplus\left([\nabla^{(\mathcal{A})}]_{G,\delta\rho}[\nu]_G(u_x)- [\nabla^{(\mathcal{A},V)}]_{G,\bar{\eta}}[v]_G(u_x)-[\omega]_G(\delta\sigma,\bar{\rho}_{\varepsilon}(u_x))\right),
\end{align*}
where $\bar{\eta}$ is such that $\delta\rho=\delta\sigma\oplus\bar{\eta}.$
From Lagrange-Poincar\'e reduction in jet bundles (see \cite[Corollary 3.2]{LP}), we know that 
$$\nabla_{\delta\sigma}^{\g}\bar{\rho}_{\varepsilon}(u_x)=\tilde{\nabla}_{u_x}\bar{\eta}-[\bar{\eta},\bar{\rho}(u_x)]+\tilde{B}(\delta\sigma,T\sigma(u_x)).$$
On the other hand, as the variation $\delta\nu$ is allowed
$$[\nabla^{(\mathcal{A})}]_{G,\delta\rho}[\nu]_G(u_x)=[\delta\nu(u_x)]_G=[\tilde{\nabla}_{u_x}\mu-[\mu,\nu(u_x)]+\omega(\delta\rho,T\rho(u_x))]_G.$$
We rewrite this class in an alternative way. First,
\begin{align*}
[\tilde{\nabla}_{u_x}\mu]_G&=[\nabla_{T\rho(u_x)}\mu]_G=[\nabla_{T\sigma(u_x)\oplus\bar{\rho}(u_x)}\mu]_G\\&=[\nabla^{(\mathcal{A},H)}]_{G,T\sigma(u_x)}[\mu]_G+[\nabla^{(\mathcal{A},V)}]_{G,\bar{\rho}(u_x)}[\mu]_G\\&=[\tilde{\nabla}^{(\mathcal{A},H)}]_{G,u_x}[\mu]_G+[\tilde{\nabla}^{(\mathcal{A},V)}]_{G,u_x}[\mu]_G,
\end{align*}
where we have defined  $[\tilde{\nabla}^{(\mathcal{A},V)}]_{G,u_x}[\mu]_G=[\nabla^{(\mathcal{A},V)}]_{G,\bar{\rho}(u_x)}[\mu]_G.$ We also have that 
$[[\mu,\nu(u_x)]]_G=[[\mu]_G,[\nu(u_x)]_G]$, and 
\begin{align*}
[\omega&(\delta\rho, T\rho(u_x))]_G=[\omega(\delta\sigma\oplus\bar{\eta},T\sigma(u_x)\oplus\bar{\rho}(u_x))]_G\\
&=[\omega]_G(\delta\sigma,T\sigma(u_x))+[\omega]_G(\bar{\eta},T\sigma(u_x))+[\omega]_G(\delta\sigma,\bar{\rho}(u_x))+[\omega]_G(\bar{\eta},\bar{\rho}(u_x)).
\end{align*}
In conclusion, for projected allowed variations 
\begin{align*}
&\frac{D^L\bar{\rho}_{\varepsilon}\oplus[\nu]_{G,\varepsilon}}{D\varepsilon}\biggr\rvert_{\varepsilon=0}(u_x)=\left(\tilde{\nabla}_{u_x}\bar{\eta}-[\bar{\eta},\bar{\rho}(u_x)]+\tilde{B}(\delta\sigma,T\sigma(u_x))\right)\oplus\\
&[\tilde{\nabla}^{(\mathcal{A},H)}]_{G,u_x}[\mu]_G+[\tilde{\nabla}^{(\mathcal{A},V)}]_{G,u_x}[\mu]_G-[[\mu]_G,[\nu(u_x)]_G]- [\nabla^{(\mathcal{A},V)}]_{G,\bar{\eta}}[v]_G(u_x)\\
&+[\omega]_G(\delta\sigma,T\sigma(u_x))+[\omega]_G(\bar{\eta},T\sigma(u_x))+[\omega]_G(\bar{\eta},\bar{\rho}(u_x)),
\end{align*}
where $\bar{\eta}\in\Gamma(\bar{U},\mathrm{Ad}P)$ such that $\pi_{\Sigma,\mathrm{Ad}P}(\bar{\eta})=\sigma$, $\bar{\eta}\vert_{\partial U}=0$ and $[\mu]_G\in\Gamma(\bar{U},V/G)$ such that $\pi_{\Sigma,V/G}([\mu]_G)=\sigma$, $[\mu]_G\vert_{\partial U}=0$. We now see that this variations coincide with the allowed variations in $J^1\Sigma\oplus (T^*X\otimes (\mathrm{Ad}Q\oplus (V/G)))$. An allowed variation of section $j^1\sigma\oplus\bar{\rho}\oplus[\nu]_G$ satisfies 
\begin{align*}
\frac{D^L\bar{\rho}_{\varepsilon}\oplus[\nu]_{G,\varepsilon}}{D\varepsilon\rvert_{\varepsilon=0}}(u_x)=\tilde{\nabla}^{\g}_{u_x}(\bar{\eta}\oplus[\mu]_G)-[\bar{\eta}\oplus[\mu]_G,\bar{\rho}\oplus[\nu]_G(u_x)]_G^{\g}+\omega^{\g}(\delta\sigma,T\sigma u_x)
\end{align*} 
where $\bar{\eta}\oplus [\mu]_G\in\Gamma(\bar{U},\mathrm{Ad}P\oplus(V/G))$ such that $\pi_{\Sigma,\mathrm{Ad}P\oplus(V/G)}(\bar{\eta}\oplus [\mu]_G)=\sigma$, and $\bar{\eta}\oplus [\mu]_G\vert_{\partial U}=0.$
The additional structure in $\mathrm{Ad}P\oplus(V/G)$ is the one detailed in Proposition \ref{quotientLP}. Hence,
\begin{align*}
\tilde{\nabla}^{\g}_{u_x}(\bar{\eta}\oplus[\mu]_G)&=\nabla^{\g}_{T\sigma(u_x)}(\bar{\eta}\oplus[\mu]_G)\\&=\nabla_{T\sigma(u_x)}\bar{\eta}\oplus \left( [\nabla^{(\mathcal{A},H)}]_{G,T\sigma(u_x)}[\mu]_G-[\omega]_G(T\sigma(u_x),\bar{\eta})\right) \\
&=\tilde{\nabla}_{u_x}\bar{\eta}\oplus \left( [\tilde{\nabla}^{(\mathcal{A},H)}]_{G,u_x}[\mu]_G+[\omega]_G(\bar{\eta},T\sigma(u_x))\right).
\end{align*}
In addition,
\begin{align*}
[\bar{\eta}\oplus[\mu]_G,\bar{\rho}\oplus[\nu]_G(u_x)]^{\g}_G&=[\bar{\eta},\bar{\rho}(u_x)]\oplus [\nabla^{(\mathcal{A},V)}]_{G,\bar{\eta}}[\nu]_G(u_x)\\-&[\nabla^{(\mathcal{A},V)}]_{G,\bar{\rho}(u_x)}[\mu]_G-[\omega]_G(\bar{\eta},\bar{\rho}(u_x))+[[\mu]_G,[\nu(u_x)]_G],
\end{align*}
and 
$$\omega^{\g}(\delta\sigma,T\sigma(u_x))=\tilde{B}(\delta\sigma,T\sigma(u_x))\oplus\omega(\delta\sigma,T\sigma(u_x)).$$
These last three expressions prove that the allowed variations in $J^1\Sigma\oplus (T^*X\otimes (\mathrm{Ad}Q\oplus (V/G)))$ are the same as the projection of the allowed variations of the original space and, consequently, (i) and (iii) are equivalent statements.
\end{proof}
\begin{remark}
Let $N$ be a normal subgroup of $G$, and $K=G/N$ the quotient group. We can reduce $L$ by $N$ and afterwards by $K$. Let $\mathcal{A}_N$ be a principal connection on $P\to P/N$ and $\mathcal{A}_{G/N}$ a principal connection on $P/N\to (P/N)/K$. These connections are said to be compatible with respect to $\mathcal{A}$ if for all $u\in TP$,
$$\mathcal{A}(u)=0 \Leftrightarrow \mathcal{A}_N(u)=0 \text{ and } \mathcal{A}_{G/N}(T\pi_{P/N,P}(u))=0.$$
In this case there exists a $\mathfrak{LP}$-isomorphism from $T(P/G)\oplus\mathrm{Ad}P\oplus(V/G)$ to $$T((P/N)/K)\oplus\tilde{\mathfrak{k}}\oplus(\tilde{\mathfrak{n}}\oplus(V/N))/(G/N),$$
where $\mathfrak{n}$ is the Lie algebra of $N$, $\mathfrak{k}$ is the Lie algebra of $K$ and $\tilde{\mathfrak{n}},\tilde{\mathfrak{k}}$ their respective adjoint bundles. Hence, it is equivalent to perform reduction directly than by stages. This will later be exemplified in \S\ref{rotor} below. For more details, see \cite[\S6.3]{CMR} and \cite[\S3.4]{BC}. 
\end{remark}

\section{Reconstruction in FT$\mathfrak{LP}$}\label{rec}
Given a critical section $j^1\sigma\oplus\bar{\rho}\oplus[\nu]_G$ in $\Gamma(\bar{U},J^1\Sigma\oplus (T^*X\otimes (\mathrm{Ad}Q\oplus (V/G))))$ for $l$, we investigate the existence of a critical section $j^1\rho\oplus\nu$ in $\Gamma(\bar{U},J^1P\oplus (T^*X\otimes V))$ for the unreduced Lagrangian $L$.

Let $\sigma(x)=\pi_{\Sigma,V/G}([\nu]_G(x))=\pi_{\Sigma,\mathrm{Ad}P}(\bar{\rho}(x))$. This section in $\Gamma(\bar{U},\Sigma)$ defines  the $G$-principal pull-back bundle $\sigma ^* P\to X$
$$\sigma^*P=\{(x,p)\vert\pi_{\Sigma,P}(p)=\sigma(x),x\in X, p\in P\}$$ of the $G$-principal bundle $P\to\Sigma$. In addition, $\sigma^*P$ can be identified with $P^{\sigma}=\{p\in P \vert \pi_{\Sigma,P}(p)\in \sigma(X)\}$ by $p\in P^{\sigma}\mapsto (\pi_{X,P}(p),p)\in\sigma^*P.$ As $\sigma(x)=\pi_{\Sigma,\mathrm{Ad}P}(\bar{\rho}(x))$, the section $\bar{\rho}(x)$ can be interpreted as a section of $\Gamma(\bar{U},\mathrm{Ad}P^{\sigma})$ and there is an equivariant horizontal 1-form $\omega^{\bar{\rho}}\in\Omega^1(P^{\sigma},\mathfrak{g})$ such that for all $x\in X$ and $u_x\in T_xX$,
$$\bar{\rho}(u_x)=[p,\omega^{\bar{\rho}}(u_p)]_G,$$
where $p\in P^{\sigma}$ and $u_p\in T_pP$ such that $T\pi_{X,P}(u_p)=u_x$. The connection $\mathcal{A}$ on $P\to \Sigma$ induces a connection $\mathcal{A}^{\sigma}$ on $P^{\sigma}\to X$ and recalling that the space of connections of a principal bundle is an affine space modeled over the space of 1-forms taking values in the adjoint bundle, we define a connection $\mathcal{A}^{\bar{\rho}}$ as:
$$\mathcal{A}^{\bar{\rho}}=\mathcal{A}^{\sigma}-\omega^{\bar{\rho}}.$$

%
%
%
\begin{theorem}
Let $\mathcal{A}$ be a principal connection on the on the principal bundle $P\to\Sigma$ and let $L:J^1P\oplus (T^*X\otimes V)\to\R$ be a $G$-invariant Lagrangian defined in a FT$\mathfrak{LP}$ bundle. Finally, let $l:J^1\Sigma\oplus (T^*X\otimes (\mathrm{Ad}Q\oplus (V/G)))\to\R$ be the reduced Lagrangian.

Then, if $j^1\rho\oplus\nu$ in $\Gamma(\bar{U},J^1P\oplus (T^*X\otimes V))$ satisfies the Lagrange-Poincar\'e equations given by $L$ in $J^1P\oplus (T^*X\otimes V)$, then the reduced section  $j^1\sigma\oplus\bar{\rho}\oplus[\nu]_G\in \Gamma(\bar{U},J^1\Sigma\oplus (T^*X\otimes (\mathrm{Ad}Q\oplus (V/G))))$ satisfies the Lagrange-Poincar\'e equations given by $l$ in $J^1\Sigma\oplus (T^*X\otimes (\mathrm{Ad}Q\oplus (V/G)))$ and connection $\mathcal{A}^{\bar{\rho}}$ on $P^{\sigma}\to\Sigma$ is flat.

Conversely, given a solution $j^1\sigma\oplus\bar{\rho}\oplus[\nu]_G\in\Gamma(\bar{U},J^1\Sigma\oplus (T^*X\otimes (\mathrm{Ad}Q\oplus (V/G))))$ of the Lagrange--Poincar\'e equations given by $l$ such that $\mathcal{A}^{\bar{\rho}}$ is flat and has trivial holonomy over an open set containing $\bar{U}$, there is a family $\Phi_g (j^1\rho\oplus\nu)$, $g \in G$, of solutions of the Lagrange--Poincar\'e equations given by $L$ projecting to $j^1\sigma\oplus\bar{\rho}\oplus[\nu]_G$. If the connection $\mathcal{A}^{\bar{\rho}}$ is flat one can always restrict it to an open simply connected set contained in $U$ so that its holonomy on $U$ is automatically zero. 
\end{theorem}
\begin{proof}
Suppose that $j^1\sigma\oplus\bar{\rho}\oplus[\nu]_G$ is the projection of $j^1\rho\oplus\nu$, in particular, $\bar{\rho}=[\rho,\rho^*\mathcal{A}]_G$, where $\rho$ is a section of $P\to X$. Observe that for all $p=\rho(x)\in P^{\sigma}$, $T_{\rho(x)}P^{\sigma}=T_x\rho(T_xX)\oplus \mathrm{ker}T_{\rho(x)}\pi_{\Sigma,P}$ and any $v_p\in T_pP^{\sigma}$ can be written as $v_p=T_x \rho(v_x)+\xi^P_p$, where $v_x\in T_xX$ and $\xi\in\mathfrak{g}$. Then, 
\begin{align*}
\mathcal{A}^{\bar{\rho}}(v_p)&=\mathcal{A}^{\sigma}(v_p)-\omega^{\bar{\rho}}(v_p)=\mathcal{A}(v_p)-\mathcal{A}(T_x\rho(T_p\pi_{X,P}(v_p)))\\&=\mathcal{A}(T_x\rho(v_x))+\mathcal{A}(\xi^P_p)-\mathcal{A}(T_x\rho(v_x))=\xi.
\end{align*}
Consequently, the horizontal subbundle defined by $\mathcal{A}^{\bar{\rho}}$ is given by $$H^{\mathcal{A}^{\bar{\rho}}}_{\rho(x)}=T_x\rho(T_xX),$$ the horizontal distribution is integrable, the integral leaves are given by $$\{\Phi_g(\rho(x))\vert x\in X,g\in G\}=\Phi_g(\mathrm{Im}\rho),$$ and $\mathcal{A}^{\bar{\rho}}$ is a flat connection on $P^{\sigma}\to X$.

Conversely, given $j^1\sigma\oplus\bar{\rho}\oplus[\nu]_G$ and $\sigma(x)=\pi_{\Sigma,V/G}([\nu]_G(x))=\pi_{\Sigma,\mathrm{Ad}P}(\bar{\rho}(x))$. Suppose that $\mathcal{A}^{\bar{\rho}}$ is flat and has trivial holonomy over an open set containing $\bar{U}$. The horizontal distribution of $\mathcal{A}^{\bar{\rho}}$ is integrable 
and the leaves cover the base. Since the holonomy is trivial each fiber intersects the leaf exactly once, that is, they are sections of $P^{\sigma}\to X$. Thus there is a family $\Phi_g(\rho(x))$ of sections of $P\to X$ that projects to $\sigma$ via $\pi_{X,\Sigma}$ and such that
$$[\rho,\rho^*\mathcal{A}]_G=[\rho,\rho^*\mathcal{A}^{\rho}+\omega^{\bar{\rho}}]_G=\bar{\rho}.$$
Furthermore, there is a unique section $\nu(x)$ of $T^*X\oplus V\to X$ such that 
$$\pi_{T^*X\oplus (V/G),T^*X\oplus V}(\nu(x))=[\nu](x),\hspace{10mm} \pi_{P,T^*X\oplus V}(\nu(x))=\rho(x).$$ 
In addition, $\Phi_g(\nu(x))$ is the unique section of $T^*X\oplus V\to X$ such that 
$$\pi_{T^*X\oplus (V/G),T^*X\oplus V}\Phi_g(\nu(x))=[\nu](x),\hspace{10mm} \pi_{P,T^*X\oplus V}\Phi_g(\nu(x))=\Phi_g(\rho(x)).$$
Thus, the family of sections $\Phi_g (j^1\rho\oplus\nu)=j^1\Phi_g(\rho)\oplus\Phi_g(\nu)$, $g \in G$, projects to $j^1\sigma\oplus\bar{\rho}\oplus[\nu]_G$ and, by equivalence (iv)$\Rightarrow$ (ii) in Proposition \ref{reduction}, they are solutions of the Lagrange--Poincar\'e equations given by $L$.
\end{proof}

The curvature of connection $\mathcal{A}^{\bar{\rho}}$  can be rewritten in terms of the curvature $B$ of $\mathcal{A}$. Then, the  flatness of $\mathcal{A}^{\bar{\rho}}$ gives the following  \textit{reconstruction condition}
\begin{equation}\label{eq:reconstruction}
B-\mathrm{d}^{\nabla^{\mathcal{A}}}\omega^{\bar{\rho}}-\omega^{\bar{\rho}}\wedge\omega^{\bar{\rho}}=0,
\end{equation} 
where $\mathrm{d}^{\nabla^{\mathcal{A}}}$ is the exterior derivative of $\mathfrak{g}$-valued forms induced by the covariant derivative on $\mathrm{Ad} P$, $\nabla^{\mathcal{A}}$, and the Cartan formula.

\section{The Noether drift law in the FT$\mathfrak{LP}$ and its reduction } \label{noether}
In this section, we define a Noether current for symmetries in FT$\mathfrak{LP}$ bundles and prove that is not a constant of motion. Instead there is a drift of this current that reduces to the new vertical equation appearing in each step of the reduction. 

\begin{definition}
Let $L:J^1P\oplus(T^*X\otimes V)\to\R$ be a Lagrangian and $G$ a Lie group acting freely and properly on $J^1P\oplus(T^*X\otimes V)$ by isomorphisms in the category FT$\mathfrak{LP}$. We define the \textit{Noether current} as the function $$J:J^1P\oplus(T^*X\otimes_P V)\to TX\otimes_P\mathfrak{g}^*$$ such that for all $j^1\rho\oplus\nu\in J^1P\oplus(T^*X\otimes_P V)$
$$J(j^1\rho\oplus\nu)((\rho,\alpha)\otimes (\rho,\eta))= \frac{\delta L}{\delta j^1\rho}(j^1\rho\oplus\nu)((\rho,\alpha)\otimes \eta^P_{\rho}),$$
where $(\rho,\alpha)\otimes (\rho,\eta)\in T^*X\otimes_P  \mathfrak{g}$.
\end{definition}
\begin{proposition}
Suppose that $L$ is $G$-invariant and let $j^1\rho\oplus\nu(x)$ be a section satisfying the Lagrange--Poincar\'e equations. Then the Noether current satisfies 
\begin{equation} \label{drift}
\mathrm{div}\left( (j^1\rho\oplus\nu)^*J(\rho(x),\eta)\right) =-\left\langle \frac{\delta L}{\delta \nu}, \omega(T_x\rho(\bullet),\eta^P_p)+\eta^V_{\nu(\bullet)}\right\rangle,
\end{equation}
where $\omega(T_x\rho(\bullet),\eta^P_p)+\eta^V_{\nu(\bullet)}\in T^*X\otimes V$ acts by replacing the slot $\bullet$ by an arbitrary element in $TX$. This is called the \emph{Noether drift law} of $L$.
\end{proposition}
\begin{proof}
As the Lagrangian is $G$-invariant, for all $\eta\in\mathfrak{g}$,
\begin{align*}
0=&\mathrm{d}L\left( \frac{d}{d\varepsilon}\biggr\rvert_{\varepsilon=0}\mathrm{exp}(\varepsilon\eta)(j^1\rho\oplus\nu)\right) =\left\langle \frac{\delta L}{\delta\rho},\frac{d}{d\varepsilon}\biggr\rvert_{\varepsilon=0}\mathrm{exp}(\varepsilon\nu)\rho\right\rangle \\
&+\left\langle \frac{\delta L}{\delta j^1\rho},\frac{D^{\nabla^{J^1P}}}{D\varepsilon}\biggr\rvert_{\varepsilon=0}\mathrm{exp}(\varepsilon\eta)j^1\rho\right\rangle+\left\langle \frac{\delta L}{\delta\nu},\frac{D^{L}}{D\varepsilon}\biggr\rvert_{\varepsilon=0}\mathrm{exp}(\varepsilon\eta)\nu\right\rangle.
\end{align*}
Since $\mathrm{exp}(\varepsilon\eta)j^1\rho=j^1(\mathrm{exp}(\varepsilon\eta)\rho)$, it is the lifted variation of the vertical variation $\mathrm{exp}(\varepsilon\eta)\rho$ of $\rho$. Then, for all $u_x\in TX$;
\begin{align*}
\frac{D^{\nabla^{J^1P}}\mathrm{exp}(\varepsilon\eta)j^1\rho}{D\varepsilon}\biggr\rvert_{\varepsilon=0}(u_x)&=\frac{D^{\nabla^P}\mathrm{exp}(\varepsilon\eta)j^1\rho(u_x)}{D\varepsilon}\biggr\rvert_{\varepsilon=0}\\&=\tilde{\nabla}^P_{u_x}\eta_{\rho}^P+T^P(\eta_{\rho}^P,T_x\rho(u_x))
\end{align*}
Furthermore, $\mathrm{exp}(\varepsilon\eta)\nu$ is a vertical variation and 
$$\frac{D^{L}\mathrm{exp}(\varepsilon\eta)\nu}{D\varepsilon}\biggr\rvert_{\varepsilon=0}(u_x)=\frac{D^{\nabla}\mathrm{exp}(\varepsilon\eta)\nu(u_x)}{D\varepsilon}\biggr\rvert_{\varepsilon=0}=\eta^V_{\nu(u_x)}.$$
Then, the $G$-invariance of $L$ can be written as,
\begin{align}\label{invariant}
0=\left\langle \frac{\delta L}{\delta\rho},\eta^
P_{\rho}\right\rangle +\left\langle \frac{\delta L}{\delta j^1\rho}, \tilde{\nabla}^P_{u_x}\eta_{\rho}^P+T^P(\eta_{\rho}^P,T_x\rho(u_x))\right\rangle+\left\langle \frac{\delta L}{\delta\nu},\eta^V_{\nu(u_x)}\right\rangle.
\end{align}
Finally, 
\begin{align*}
\mathrm{div}&\left( (j^1\rho\oplus\nu)^*J(\rho(x),\eta)\right)=\mathrm{div}\left\langle \frac{\delta L}{\delta j^1\rho}(j^1\rho\oplus\nu(x)),\eta^P_{\rho(x)}\right\rangle \\&=\left\langle \mathrm{div}^P\frac{\delta L}{\delta j^1\rho},\eta^P_{\rho}\right\rangle +\left\langle \frac{\delta L}{\delta j^1\rho},\tilde{\nabla}^P\eta^P_{\rho}\right\rangle \\
&=\left\langle \mathrm{div}^P\frac{\delta L}{\delta j^1\rho},\eta^P_{\rho}\right\rangle-\left\langle \frac{\delta L}{\delta\rho},\eta^P_{\rho}\right\rangle -\left\langle \frac{\delta L}{\delta j^1\rho}, T^P(\eta_{\rho}^P,T_x\rho(\bullet))\right\rangle-\left\langle \frac{\delta L}{\delta\nu},\eta^V_{\nu(\bullet)}\right\rangle \\
&=-\left\langle \frac{\delta L}{\delta\nu},\omega(T\rho(\bullet),\eta^P_p)+\eta^V_{\nu(\bullet)}\right\rangle, 
\end{align*}
where we have used relation \ref{invariant} and Lagrange--Poincar\'e equations.
\end{proof}

The Noether current is $G$-equivariant so that it defines a bundle map in the quotient by $G$
$$j:J^1\Sigma\oplus(T^*X\otimes_\Sigma (\mathrm{Ad}P\oplus(V/G)))\to TX\otimes_P\mathrm{Ad}P^*$$ such that for all $j^1\sigma\oplus\bar{\rho}\oplus[\nu]_G\in J^1\Sigma\oplus(T^*X\otimes_\Sigma (\mathrm{Ad}P\oplus(V/G)))$, and all $(\sigma,\alpha)\otimes \bar{\eta}\in T^*X\otimes_{\Sigma} \mathrm{Ad}P$,
\begin{align*}
&j(j^1\sigma\oplus\bar{\rho}\oplus[\nu]_G)((\sigma,\alpha)\otimes \bar{\eta})=\\&= J(j^1\rho\oplus\nu)((\rho,\alpha)\otimes (\rho,\eta))=\frac{d}{d\epsilon}\biggr\rvert_{\epsilon=0}L(j^1\rho+\epsilon((\rho,\alpha)\otimes\eta^P_{\rho})\oplus\nu)\\&=\frac{d}{d\epsilon}\biggr\rvert_{\epsilon=0}l(j^1\sigma\oplus(\bar{\rho}+\epsilon((\sigma,\alpha)\otimes\bar{\eta})\oplus[\nu]_G)=\frac{\delta l}{\delta\bar{\rho}}(j^1\sigma\oplus\bar{\rho}\oplus[\nu]_G)((\sigma,\alpha)\otimes \bar{\eta}).
\end{align*} 
Consequently, the reduced Noether current $j$ coincides with
$\frac{\delta l}{\delta\bar{\rho}}.$
\begin{proposition}
The drift of the Noether current along critical sections $j^1\sigma\oplus\bar{\rho}\oplus[\nu]_G$ given by equation \ref{drift} projects to the equation
\begin{equation}
\left\langle \mathrm{div}^{\mathcal{A}}\frac{\delta l}{\delta\bar{\rho}},\bar{\eta}\right\rangle =\left\langle \mathrm{ad}^*_{\bar{\rho}}\frac{\delta l}{\delta\bar{\rho}},\bar{\eta}\right\rangle -\left\langle \frac{\delta l}{\delta[\nu]_G},[\nabla^{(\mathcal{A},V)}]_{G,\bar{\eta}}[\nu]_G+[\omega]_G(T\sigma\oplus\bar{\rho},\bar{\eta})\right\rangle,
\end{equation}\label{newLP}
where $\mathrm{div}^{\mathcal{A}}$ is the divergence of $\mathrm{Ad}P^*$-valued vector fields induced by connection $\nabla^{\mathcal{A}}$ in $\mathrm{Ad}P$ and $\bar{\eta}(x)=[\rho(x),\eta]_G$
\end{proposition}
\begin{proof}
We rewrite the left hand side of equation \ref{drift}:
\begin{align*}
 \mathrm{div}&\left\langle J(j^1\rho\oplus\nu),(\rho,\eta)\right\rangle =\mathrm{div}\left\langle j(j^1\sigma\oplus\bar{\rho}\oplus[\nu]_G),\bar{\eta}\right\rangle\\
&=\left\langle \mathrm{div}^{\mathcal{A}}j(j^1\sigma\oplus\bar{\rho}\oplus[\nu]_G),\bar{\eta}\right\rangle+\left\langle j(j^1\sigma\oplus\bar{\rho}\oplus[\nu]_G),\tilde{\nabla}^{\mathcal{A}}\bar{\eta}\right\rangle\\
&=\left\langle \mathrm{div}^{\mathcal{A}}\frac{\delta l}{\delta\bar{\rho}}(j^1\sigma\oplus\bar{\rho}\oplus[\nu]_G),\bar{\eta}\right\rangle-\left\langle \frac{\delta l}{\delta\bar{\rho}}(j^1\sigma\oplus\bar{\rho}\oplus[\nu]_G),[\bar{\rho},\bar{\eta}]\right\rangle\\
&=\left\langle \mathrm{div}^{\mathcal{A}}\frac{\delta l}{\delta\bar{\rho}}(j^1\sigma\oplus\bar{\rho}\oplus[\nu]_G)-\mathrm{ad}^*_{\bar{\rho}}\frac{\delta l}{\delta\bar{\rho}}(j^1\sigma\oplus\bar{\rho}\oplus[\nu]_G),\bar{\eta}\right\rangle.
\end{align*}
On the other hand, the right hand side of equation \ref{drift} projects to
$$ -\left\langle \frac{\delta l}{\delta[\nu]_G},[\omega]_G(T\sigma(\bullet)\oplus\bar{\rho}(\bullet),\bar{\eta})+[\eta^V_{\nu(\bullet)}]_G\right\rangle,$$
and we conclude by observing that $[\eta^V_{\nu(\bullet)}]_G=[\nabla^{(\mathcal{A},V)}]_{G,\bar{\eta}}[\nu]_G(\bullet).$
\end{proof}

The vertical Lagrange-Poincar\'e equations, 
$$\mathrm{div}^{\g}\left( \frac{\delta l}{\delta\bar{\rho}}\oplus\frac{\delta l}{\delta[\nu]_G}\right) -\mathrm{ad}^*_{\bar{\rho}\oplus[\nu]_G}\left( \frac{\delta l}{\delta\bar{\rho}}\oplus\frac{\delta l}{\delta[\nu]_G}\right)=0,$$ associated to the reduced bundle $J^1\Sigma\oplus(T^*X\oplus(\mathrm{Ad}P\oplus (V/G)))$ is an equation in $\mathrm{Ad}P^*\oplus(V/G)^*$ acting on vectors $\bar{\eta}\oplus[u]_G\in\mathrm{Ad}P\oplus(V/G)$. We can decompose these equations restricting to each of the factors on $\mathrm{Ad}P\oplus(V/G)$. First, 
\begin{align*}
&\left\langle \mathrm{div}^{\g}\left( \frac{\delta l}{\delta\bar{\rho}}\oplus\frac{\delta l}{\delta[\nu]_G}\right),\bar{\eta}\oplus[u]_G\right\rangle \\&=\mathrm{div}\left\langle \left( \frac{\delta l}{\delta\bar{\rho}}\oplus\frac{\delta l}{\delta[\nu]_G}\right),\bar{\eta}\oplus[u]_G\right\rangle-\left\langle \left( \frac{\delta l}{\delta\bar{\rho}}\oplus\frac{\delta l}{\delta[\nu]_G}\right),\tilde{\nabla}^{\g}(\bar{\eta}\oplus[u]_G)\right\rangle \\
&=\mathrm{div}\left\langle \left( \frac{\delta l}{\delta\bar{\rho}}\oplus\frac{\delta l}{\delta[\nu]_G}\right),\bar{\eta}\oplus[u]_G\right\rangle-\left\langle \frac{\delta l}{\delta\bar{\rho}},\tilde{\nabla}^{\mathcal{A}}\bar{\eta}\right\rangle -\left\langle \frac{\delta l}{\delta[\nu]_G},[\tilde{\nabla}^{(\mathcal{A},H)}]_G[u]_G\right\rangle\\ &\phantom{=}+\left\langle \frac{\delta l}{\delta[\nu]_G},[\omega]_G(T\sigma,\bar{\eta})\right\rangle\\
&=\mathrm{div}\left\langle \left( \frac{\delta l}{\delta\bar{\rho}}\oplus\frac{\delta l}{\delta[\nu]_G}\right),\bar{\eta}\oplus[u]_G\right\rangle-\mathrm{div}\left\langle\frac{\delta l}{\delta\bar{\rho}},\bar{\eta}\right\rangle +\left\langle \mathrm{div}^{\mathcal{A}}\frac{\delta l}{\delta\bar{\rho}},\bar{\eta}\right\rangle \\&\phantom{=}-\mathrm{div}\left\langle\frac{\delta l}{\delta[\nu]_G},[u]_G\right\rangle +\left\langle \mathrm{div}^{(\mathcal{A},H)}\frac{\delta l}{\delta [\nu]_G},[u]_G\right\rangle+\left\langle \frac{\delta l}{\delta[\nu]_G},[\omega]_G(T\sigma,\bar{\eta})\right\rangle \\
&=\left\langle \mathrm{div}^{\mathcal{A}}\frac{\delta l}{\delta\bar{\rho}},\bar{\eta}\right\rangle+\left\langle \mathrm{div}^{(\mathcal{A},H)}\frac{\delta l}{\delta [\nu]_G},[u]_G\right\rangle+\left\langle \frac{\delta l}{\delta[\nu]_G},[\omega]_G(T\sigma,\bar{\eta})\right\rangle 
\end{align*}
On the other hand;
\begin{align*}
&\left\langle \mathrm{ad}^*_{\bar{\rho}\oplus[\nu]_G}\left( \frac{\delta l}{\delta\bar{\rho}}\oplus\frac{\delta l}{\delta[\nu]_G}\right),\bar{\eta}\oplus[u]_G\right\rangle=\left\langle \frac{\delta l}{\delta\bar{\rho}}\oplus\frac{\delta l}{\delta[\nu]_G},[\bar{\rho}\oplus[\nu]_G,\bar{\eta}\oplus[u]_G]\right\rangle\\
&\phantom{=}= \left\langle \frac{\delta l}{\delta\bar{\rho}},[\bar{\rho},\bar{\eta}]\right\rangle \\&\phantom{==}+\left\langle \frac{\delta l}{\delta[\nu]_G},
[\nabla^{(\mathcal{A},V)}]_{G,\bar{\rho}}[u]_G-[\nabla^{(\mathcal{A},V)}]_{G,\bar{\eta}}[\nu]_G-[\omega]_G(\bar{\rho},\bar{\eta})+[[\nu]_G,[u]_G]\right\rangle\\
&\phantom{=}= \left\langle \mathrm{ad}^*_{\bar{\rho}}\frac{\delta l}{\delta\bar{\rho}},\bar{\eta}\right\rangle +\left\langle \mathrm{ad}^*_{[\nu]_G}\frac{\delta l}{\delta[\nu]_G},[u]_G\right\rangle\\&\phantom{==} +\left\langle \frac{\delta l}{\delta[\nu]_G},[\nabla^{(\mathcal{A},V)}]_{G,\bar{\rho}}[u]_G-[\nabla^{(\mathcal{A},V)}]_{G,\bar{\eta}}[\nu]_G-[\omega]_G(\bar{\rho},\bar{\eta})\right\rangle. 
\end{align*}
Thus, the vertical Lagrange--Poincar\'e equations restricted to $\mathrm{Ad}P$ are
$$\left\langle \mathrm{div}^{\mathcal{A}}\frac{\delta l}{\delta\bar{\rho}},\bar{\eta}\right\rangle=\left\langle \mathrm{ad}^*_{\bar{\rho}}\frac{\delta l}{\delta\bar{\rho}},\bar{\eta}\right\rangle -\left\langle \frac{\delta l}{\delta[\nu]_G},[\nabla^{(\mathcal{A},V)}]_{G,\bar{\eta}}[\nu]_G+[\omega]_G(T\sigma\oplus\bar{\rho},\bar{\eta})\right\rangle,$$
which according to Proposition \ref{drift}  coincides with the drift of the Noether current, while the vertical Lagrange--Poincar\'e equations restricted to $V/G$ are 
$$\left\langle \mathrm{div}^{(\mathcal{A},H)}\frac{\delta l}{\delta [\nu]_G},[u]_G\right\rangle=\left\langle \mathrm{ad}^*_{[\nu]_G}\frac{\delta l}{\delta[\nu]_G},[u]_G\right\rangle +\left\langle \frac{\delta l}{\delta[\nu]_G},[\nabla^{(\mathcal{A},V)}]_{G,\bar{\rho}}[u]_G\right\rangle, $$ 
obtained from the projections of the vertical Lagrange-Poincar\'e equations induced by $L$.
\begin{remark}
In the special case $V=0$, the drift law becomes a conservation law expressed as a vanishing of a divergence. Indeed, we recover the Noether theorem for covariant invariant Lagrangians. In addition, the conservation of this current is equivalent to the vertical Lagrange-Poincar\'e equation of the reduced Lagrangian:
$$ \mathrm{div}^{\mathcal{A}}\frac{\delta l}{\delta\bar{\rho}}- \mathrm{ad}^*_{\bar{\rho}}\frac{\delta l}{\delta\bar{\rho}}=0.$$
\end{remark}

\section{The molecular strand with rotors}\label{rotor}
In this section we will discuss a problem of strand dynamics as an example of the theory of reduction by stages. Our model is called the molecular strand with rotors and consists of a mobile base strand repeating the same configuration with different orientation. In turn, this configuration will have a moving piece that can rotate, called rotor. In this example we will suppose that each configuration has three rotors, each in the principal axes. However, the case with only one rotor or multiple rotors non-necessarily along the principal axes follows directly from our description. The   particular case of this model when there are no rotors has been the object of research, for example, in \cite{Mol Strand} following the same covariant Lagrangian approach.

The principal bundle that we shall use has $X=\R^2$ as the base space and $P=X\times\R^3\times SO(3)\times \sS^1\times \sS^1\times \sS^1$ as the total space. The variables in $\mathbb{R}^2$ willl be denoted by $(t,s)$, the first one can be thought as the time whereas the second is the parameter of the stradt. The Lie group $SO(3)\times \sS^1\times \sS^1\times \sS^1$ acts on $P$ since for any $(x,r,\Lambda,\theta)\in X\times\R^3\times SO(3)\times \sS^1\times \sS^1\times \sS^1$, where $x=(s,t)\in\R^2$ and $\theta=(\theta_1,\theta_2,\theta_3)\in\sS^1\times \sS^1\times \sS^1$, the element $(\Gamma,\alpha)\in SO(3)\times \sS^1\times \sS^1\times \sS^1$ acts by
$$(\Gamma,\alpha)\cdot(x,r,\Lambda,\theta)=(x,\Gamma r,\Gamma\Lambda,\theta+\alpha).$$
We will first reduce by the normal subgroup $SO(3)$ and then by $\sS^1\times \sS^1\times \sS^1$. 

\subsection{First reduction} We denote $\Sigma=P/SO(3)=X\times\R^3\times\sS^1\times \sS^1\times \sS^1$, the first quotient space, and $(x,\rho,\theta)\in\Sigma$. The projection is given by 
\begin{align*}
\pi_{\Sigma,P}:X\times\R^3\times SO(3)\times \sS^1\times \sS^1\times \sS^1\to &X\times\R^3\times SO(3)\times \sS^1\times \sS^1\times \sS^1\\
(x,r,\Lambda,\theta)\mapsto &(x,\rho=\Lambda^{-1}r,\theta)
\end{align*}
and its derivative is
\begin{align*}
T\pi_{\Sigma,P}:TP\to &T\Sigma\\
(x,r,\Lambda,\theta,v_x,v_r,v_{\Lambda},v_{\theta})\mapsto &(x,\rho=\Lambda^{-1}r,\theta,v_x,v_{\rho}=\Lambda^{-1}v_r-\Lambda^{-1}v_{\Lambda}\rho,v_{\theta}).
\end{align*}
To identify $J^1P/SO(3)$ as a FT$\mathfrak{LP}$ bundle we need a connection $\mathcal{A}$ on $P\to\Sigma$. In terms of the Maurer-Cartan connection, any connection $\mathcal{A}$ can be written as
$$\mathcal{A}(x,r,\Lambda,\theta,v_x,v_r,v_{\Lambda},v_{\theta})=v_{\Lambda}\Lambda^{-1}+\tilde{\mathcal{A}}(x,r,\Lambda,\theta,v_x,v_r,v_{\Lambda},v_{\theta})$$
such that for any $\eta\in \so (3)$, $\tilde{\mathcal{A}}(\eta^P_{x,r,\Lambda,\theta})=0$ and for all $v\in TP$, $\tilde{\mathcal{A}}(\Gamma v)=\mathrm{Ad}_{\Gamma}\tilde{\mathcal{A}}(v)$. In addition, since $P\to X$ is trivial, we can trivialize $\mathrm{Ad}P$ as $\Sigma\times\so (3)$ via the identification $$[(x,r,\Lambda,\theta),\zeta]_{SO(3)}=[(x,\Lambda^{-1}r,e,\theta),\mathrm{Ad}_{\Lambda^{-1}}\zeta]_{SO(3)}=(x,\Lambda^{-1}r,\theta,\mathrm{Ad}_{\Lambda^{-1}}\zeta).$$ Then, the covariant derivative in  $\mathrm{Ad}P$
\begin{align*}
&\frac{D^{\mathcal{A}}}{D\tau}[x(\tau),r(\tau),\Lambda(\tau),\theta(\tau),\zeta(\tau)]_{SO(3)}=\\&[x(\tau),r(\tau),\Lambda(\tau),\theta(\tau),\dot{\zeta}(\tau)-(\Lambda_{\tau}\Lambda^{-1}+\tilde{\mathcal{A}}(x,r,\Lambda,\theta,x_{\tau},r_{\tau},\Lambda_{\tau},\theta_{\tau}))\times\dot{\zeta}(\tau)]_{SO(3)}
\end{align*}
induces in $\Sigma\times\so (3)$ the covariant derivative
\begin{multline*}
\frac{D^{\mathcal{A}}}{D\tau}(x(\tau),\rho(\tau),\theta(\tau),\zeta(\tau))=\\=(x(\tau),\rho(\tau),\theta(\tau),\dot{\zeta}(\tau)-\tilde{\mathcal{A}}(x,\rho,e,\theta,x_{\tau},\rho_{\tau},0,\theta_{\tau})\times\dot{\zeta}(\tau)).
\end{multline*}
We lift the section $(x,r(x),\Lambda(x),\theta(x))$ of $P\to X$ to the section in $J^1P\to X$,
$$(x,r(x),\Lambda(x),\theta(x)),ds+dt,r_sds+r_tdt,\Lambda_sds+\Lambda_tdt,\theta_sds+\theta_tdt).$$ 
To project this section to $J^1\Sigma\oplus(T^*X\otimes(\Sigma\times\so(3)))$ we evaluate
\begin{align*}
&\mathcal{A}(x,r,\Lambda,\theta,ds+dt,r_sds+r_tdt,\Lambda_sds+\Lambda_tdt,\theta_sds+\theta_tdt)=\\
=&\Lambda_s\Lambda^{-1}ds+\Lambda_t\Lambda^{-1}dt+\tilde{\mathcal{A}}(x,r,\Lambda,\theta,1,r_s,\Lambda_s,\theta_s)ds+\tilde{\mathcal{A}}(x,r,\Lambda,\theta,1,r_t,\Lambda_t,\theta_t)dt
\end{align*}
and 
\begin{align*}
&\mathrm{Ad}_{\Lambda^{-1}}\mathcal{A}(x,r,\Lambda,\theta,ds+dt,r_sds+r_tdt,\Lambda_sds+\Lambda_tdt,\theta_sds+\theta_tdt)=\\
&=\Omega ds+\omega dt+\bar{\mathcal{A}}_sds+\bar{\mathcal{A}}_t dt=\Xi ds +\xi dt,
\end{align*}
where $\Omega=\Lambda^{-1}\Lambda_s$, $\omega=\Lambda^{-1}\Lambda_t$,$$\bar{\mathcal{A}}_s=\mathrm{Ad}_{\Lambda^{-1}}\tilde{\mathcal{A}}(x,r,\Lambda,\theta,1,r_s,\Lambda_s,\theta_s)=\tilde{\mathcal{A}}(x,r,\Lambda,\theta,1,\rho_s,0,\theta_s),$$ $\bar{\mathcal{A}}_t=\tilde{\mathcal{A}}(x,r,\Lambda,\theta,1,\rho_t,0,\theta_t)$, $\Xi=\Omega +\bar{\mathcal{A}}_s$ and $\xi=\omega +\bar{\mathcal{A}}_t.$ Consequently, the bundle $J^1P/SO(3)$ is identified with FT$\mathfrak{LP}$ bundle $J^1\Sigma\oplus(T^*X\otimes(\Sigma\times\so(3)))$ via
\begin{align*}
\pi_{\so(3)}:J^1P/G\to &J^1\Sigma\oplus(T^*X\otimes(\Sigma\times\so(3)))\\
j^1(x,r(x),\Lambda(x),\theta(x))\mapsto&j^1(x,\rho(x),\theta(x))\oplus(x,\rho(x),\theta(x),\Xi(x)ds+\xi(x)dt)
\end{align*}
and the reduced variables are $\rho$, $\rho_s$, $\rho_t$, $\theta$, $\theta_s$, $\theta_t$, $\Xi$ and $\xi$.
The vertical Lagrange-Poincar\'e equations are:
\begin{equation*}
\mathrm{ad}^*_{\Xi ds+\xi dt}\frac{\delta l}{\delta(\Xi ds+\xi dt)}-\mathrm{div}^{\nabla}\frac{\delta l}{\delta(\Xi ds+\xi dt)}=0.
\end{equation*}
Hence, since 
$$\mathrm{ad}^*_{\Xi ds+\xi dt}\frac{\delta l}{\delta(\Xi ds+\xi dt)}=-\Xi\times\frac{\delta l}{\delta\Xi}-\xi\times\frac{\delta l}{\delta\xi},$$
and
$$\mathrm{div}^{\nabla}\frac{\delta l}{\delta(\Xi ds+\xi dt)}=\partial_s\frac{\delta l}{\delta\Xi}-\bar{\mathcal{A}}_s\times\frac{\delta l}{\delta\Xi}+\partial_t\frac{\delta l}{\delta\xi}-\bar{\mathcal{A}}_t\times\frac{\delta l}{\delta\xi},$$
we conclude that the vertical Lagrange--Poincar\'e equations are
\begin{equation}\label{LP1vert}
0=\partial_s\frac{\delta l}{\delta\Xi}+\partial_t\frac{\delta l}{\delta\xi}+\Omega\times\frac{\delta l}{\delta\Xi}+\omega\times\frac{\delta l}{\delta\xi}.
\end{equation}
On the other hand, the horizontal Lagrange--Poincar\'e equations applied to variation $\delta\rho\oplus\delta\theta$ are
\begin{multline*}
\left\langle \frac{\delta l}{\delta(\rho\oplus\theta)}-\partial_s\frac{\delta l}{\delta(\rho_s\oplus\theta_s)}-\partial_t\frac{\delta l}{\delta(\rho_t\oplus\theta_t)},\delta\rho\oplus\delta\theta\right\rangle=\\=\left\langle \frac{\delta l}{\delta(\Xi ds+\xi dt)},\tilde{B}((ds+dt,\rho_sds+\rho_tdt,\theta_sds+\theta_tdt),(0,\delta\rho,\delta\theta))\right\rangle  .
\end{multline*}
We shall note that $\delta l/\delta(\rho\oplus\theta)$ is not a partial derivative, instead we saw in section \ref{FTLPmec} that 
$$\left\langle \frac{\delta l}{\delta(\rho\oplus\theta)},\delta\rho\oplus\delta\theta\right\rangle=\frac{d}{d\epsilon}\biggr\vert_{\epsilon=0}l(u^h_{j^1\rho\oplus j^1\theta,\Xi ds+\xi dt}),$$
where $u(\epsilon)$ is a curve in $\Sigma=X\times\R^3\times\sS^1\times\sS^1\times\sS^1$ with $u'(0)=(0,\delta\rho,\delta\theta)$ and $u^h_{j^1\rho\oplus j^1\theta,\Xi ds+\xi dt}(\epsilon)=(u(\epsilon),j^1\rho(0),j^1\theta(0),\nu(\epsilon))$ the horizontal lift to $J^1\Sigma\oplus(T^*X\otimes(\Sigma\times\so(3)))$ with $\nu(\epsilon)=\nu_1(\epsilon)ds+\nu_2(\epsilon)dt$ such that $v(0)=\Xi ds+\xi dt$ and $\nu(\epsilon)$ is horizontal. From the connection in $\Sigma\times\so(3)$ this means that $$\dot{\nu}_1(\epsilon)ds+\dot{\nu}_2(\epsilon)dt=\mathcal{A}(0,\delta\rho,0,\delta\theta)\times(\nu_1(\epsilon)ds+\nu_2(\epsilon)dt).$$ In particular, for $\epsilon=0$;
$$\dot{\nu}_1(0)ds+\dot{\nu}_2(0)dt=\mathcal{A}(0,\delta\rho,0,\delta\theta)\times(\Xi ds+\xi dt),$$
and 
\begin{multline*}
\left\langle \frac{\delta l}{\delta(\rho\oplus\theta)},\delta\rho\oplus\delta\theta\right\rangle=\left\langle \frac{\partial l}{\partial\rho},\delta\rho\right\rangle +\left\langle \frac{\partial l}{\partial\theta},\delta\theta\right\rangle \\+\left\langle \frac{\delta l}{\delta\Xi},\tilde{\mathcal{A}}(0,\delta\rho,0,\delta\theta)\times\Xi\right\rangle +\left\langle \frac{\delta l}{\delta\xi},\tilde{\mathcal{A}}(0,\delta\rho,0,\delta\theta)\times\xi\right\rangle. 
\end{multline*}
It is possible to split the horizontal equations in $\Sigma$ in two parts. One related to $\R^3$ for which we make $\delta\theta=0$;
\begin{align}\label{LP1hor1}
&\left\langle \frac{\partial l}{\partial\rho}-\partial_s\frac{\delta l}{\partial\rho_s}-\partial_t\frac{\delta l}{\partial\rho_t},\delta\rho\right\rangle+\left\langle \frac{\delta l}{\delta\Xi},\tilde{\mathcal{A}}(0,\delta\rho,0,0)\times\Xi\right\rangle+\\&\left\langle \frac{\delta l}{\delta\xi},\tilde{\mathcal{A}}(0,\delta\rho,0,0)\times\xi\right\rangle =\left\langle \frac{\delta l}{\delta\Xi},\tilde{B}((1,\rho_s,\theta_s),\delta\rho)\right\rangle+\left\langle \frac{\delta l}{\delta\xi},\tilde{B}((1,\rho_t,\theta_t),\delta\rho)\right\rangle,\nonumber
\end{align}
and one related to $\sS^1\times\sS^1\times\sS^1$ for which $\delta\rho=0$;
\begin{align}\label{LP1hor2}
&\left\langle \frac{\partial l}{\partial\theta}-\partial_s\frac{\delta l}{\partial\theta_s}-\partial_t\frac{\delta l}{\partial\theta_t},\delta\theta\right\rangle+\left\langle \frac{\delta l}{\delta\Xi},\tilde{\mathcal{A}}(0,0,0,\delta\theta)\times\Xi\right\rangle+\\&\left\langle \frac{\delta l}{\delta\xi},\tilde{\mathcal{A}}(0,0,0,\delta\theta)\times\xi\right\rangle =\left\langle \frac{\delta l}{\delta\Xi},\tilde{B}((1,\rho_s,\theta_s),\delta\theta)\right\rangle+\left\langle \frac{\delta l}{\delta\xi},\tilde{B}((1,\rho_t,\theta_t),\delta\theta)\right\rangle\nonumber
\end{align}
In conclusion, the equations of motion of the molecular strand after reduction by $SO(3)$ are equations \eqref{LP1vert}, \eqref{LP1hor1} and \eqref{LP1hor2}. Yet, they are only equivalent to the unreduced equations if considered together with the reconstruction equation studied in \S\ref{rec} above. Observe that in $\mathrm{Ad}P$, $\omega^{\bar{\rho}}=\Lambda\Xi\Lambda^{-1}ds+\Lambda\xi\Lambda^{-1}dt$. Therefore,
$$d^{\mathcal{A}}\omega^{\bar{\rho}}=\mathrm{Ad}_{\Lambda}(\Xi_t-\bar{\mathcal{A}}_t\times\Xi-\xi_s+\bar{\mathcal{A}}_s\times\xi)dt\wedge ds,$$
and $$\omega^{\bar{\rho}}\wedge\omega^{\bar{\rho}}=(\Lambda\Xi\Lambda^{-1}ds+\Lambda\xi\Lambda^{-1}dt)\wedge(\Lambda\Xi\Lambda^{-1}ds+\Lambda\xi\Lambda^{-1}dt)=-\mathrm{Ad}_{\Lambda}(\Xi\times\xi)dt\wedge ds.$$
The reconstruction equation (\ref{eq:reconstruction}) is then written as
\begin{equation}
\xi_s-\bar{\mathcal{A}}_s\times\xi-\Xi_t+\bar{\mathcal{A}}_t\times\Xi-\Xi\times\xi+\mathrm{Ad}_{\Lambda}\tilde{B}((0,\rho_t,\theta_t,0,\rho_s,\theta_s))=0.
\end{equation}

\subsection{Second reduction} The group $S=\sS^1\times\sS^1\times\sS^1$ acts on $J^1\Sigma\oplus(T^*X\otimes(\Sigma\times\so(3)))$ by
\begin{align*}
\alpha\cdot(x,\rho,\theta,ds+dt,\rho_sds+\rho_tdt,\theta_sds+\theta_tdt)\oplus(x,\rho,\theta,\Xi ds+\xi dt)\\=(x,\rho,\theta+\alpha,ds+dt,\rho_sds+\rho_tdt,\theta_sds+\theta_tdt)\oplus(x,\rho,\theta+\alpha,\Xi ds+\xi dt)
\end{align*}
for all $\alpha\in \sS^1\times\sS^1\times\sS^1$. Let $R=\Sigma/(\sS^1\times\sS^1\times\sS^1)=X\times\R^3$. Since $\sS^1\times\sS^1\times\sS^1$ is abelian, the identification $\mathrm{Ad}\Sigma=R\times \R^3$ is immediate and the principal connection in $\Sigma\to R$ is trivial. Consequently, we have the identification
\begin{align*}
\pi_S:J^1\Sigma\oplus(T^*X&\otimes(\Sigma\times\so(3)))/S\to J^1R\oplus(T^*X\otimes(R\times\R^3\times\so(3)))\\
[(x,\rho,\theta,ds+dt,\rho_sds&+\rho_tdt,\theta_sds+\theta_tdt)\oplus(x,\rho,\theta,\Xi ds+\xi dt)]_S\mapsto\\&(x,\rho,ds+dt,\rho_sds+\rho_tdt)\oplus(x,\rho,\theta_sds+\theta_tdt,\Xi ds+\xi dt),
\end{align*}   
where $J^1R\oplus(T^*X\otimes(R\times\R^3\times\so(3)))$ is a FT$\mathfrak{LP}$ bundle with extra structure given by;
\begin{multline*}\nabla^{\mathfrak{s}}_{v_x,v_{\rho}}(x,\rho,\beta(x,\rho),\zeta(x,\rho))=\\=(x,\rho,\beta_{v_x,v_{\rho}},\zeta_{v_x,v_{\rho}}-\tilde{B}((v_x,v_{\rho},0),(0,0,\beta))-\tilde{\mathcal{A}}(v_x,v_{\rho},0,0)\times\zeta);
\end{multline*}
the $\R^3\times \so(3)$-valued $2$-form $\omega^{\mathfrak{s}}=0\oplus[\tilde{B}]_S$; and the Lie bracket
\begin{align*}
[(&\beta_1,\zeta_1),(\beta_2,\zeta_2)]^{\mathfrak{s}}=0\oplus\\&(-\mathcal{A}(0,0,0,\beta_1)\times\zeta_2+\mathcal{A}(0,0,0,\beta_2)\times\zeta_1-\tilde{B}((0,0,\beta_1),(0,0,\beta_2))+\zeta_1\times\zeta_2).
\end{align*}
The vertical Lagrange--Poincar\'e equations in the second step of reduction are
\begin{equation*}
\mathrm{ad}^*_{(\theta_s,\Xi)ds+(\theta_t,\xi)dt}\frac{\delta l}{\delta(\theta_sds+\theta_tdt,\Xi ds+\xi dt)}-\mathrm{div}^{\mathfrak{s}}\frac{\delta l}{\delta(\theta_sds+\theta_tdt,\Xi ds+\xi dt)}=0.
\end{equation*}
Given $(\delta\theta,\delta\eta)\in\R^3\times\so(3)$, we obtain that, on one hand,
\begin{align*}
&\left\langle \mathrm{ad}^*_{(\theta_s,\Xi)ds+(\theta_t,\xi)dt}\frac{\delta l}{\delta(\theta_sds+\theta_tdt,\Xi ds+\xi dt)},(\delta\theta,\delta\eta)\right\rangle \\&=\left\langle \frac{\delta l}{\delta(\theta_sds+\theta_tdt)}\oplus\frac{\delta l}{\delta(\Xi ds+\xi dt)},[(\theta_sds+\theta_tdt,\Xi ds+\xi dt),(\delta\theta,\delta\eta)]\right\rangle \\&=\left\langle \frac{\delta l}{\delta\Xi},-\mathcal{A}(0,0,0,\theta_s)\times\delta\eta+\mathcal{A}(0,0,0,\delta\theta)\times\Xi-\tilde{B}((0,0,\theta_s),(0,0,a))+\Xi\times\delta\eta\right\rangle \\
&\phantom{=}+\left\langle \frac{\delta l}{\delta\xi},-\mathcal{A}(0,0,0,\theta_t)\times\delta\eta+\mathcal{A}(0,0,0,\delta\theta)\times\Xi-\tilde{B}((0,0,\theta_t),(0,0,a))+\xi\times\delta\eta\right\rangle
\end{align*}
On the other hand,
\begin{align*}
&\left\langle \mathrm{div}^{\mathfrak{s}}\frac{\delta l}{\delta(\theta_sds+\theta_tdt,\Xi ds+\xi dt)},(\delta\theta,\delta\eta)\right\rangle=\bigg(
\left\langle \partial_s\frac{\delta l}{\delta\theta_s}+\partial_t\frac{\delta l}{\delta\theta_t},\delta\theta\right\rangle \\
 &\phantom{==}+\left\langle \frac{\delta l}{\delta\Xi},\tilde{B}((0,\rho_s,0),(0,0,\delta\theta))\right\rangle +\left\langle\frac{\delta l}{\delta\xi},\tilde{B}((0,\rho_t,0),(0,0,\delta\theta))\right\rangle \bigg)\\
 &\phantom{=}\oplus\left\langle \partial_s\frac{\delta l}{\delta\Xi}+\partial_t\frac{\delta l}{\delta\xi}+\mathcal{A}(0,\rho_s,0,0)\times\frac{\delta l}{\delta\Xi}+\mathcal{A}(0,\rho_t,0,0)\times\frac{\delta l}{\delta\xi},\delta\eta\right\rangle.
\end{align*}
From these two expressions, it can be proven that the vertical Lagrange-Poincar\'e equations can be split into two parts related to $\R^3$ and $\so(3)$. These respectively coincide with equations \eqref{LP1hor2} and \eqref{LP1vert}, except for the term $\delta l/\delta\theta$ which is zero from the symmetry of the Lagrangian with respect to $\sS^1\times\sS^1\times\sS^1$.
The horizontal Lagrange--Poincar\'e equations are
\begin{align*}
\left\langle \frac{\delta l}{\delta\rho}-\partial_s\frac{\delta l}{\delta\rho_s}-\partial_t\frac{\delta l}{\delta\rho_t},\delta\rho\right\rangle =\left\langle \frac{\delta l}{\delta\Xi},\tilde{B}((0,\rho_s,0),\delta\rho)\right\rangle +\left\langle \frac{\delta l}{\delta\xi},\tilde{B}((0,\rho_t,0),\delta\rho)\right\rangle. 
\end{align*}
In an analogous way as we did in the first step of reduction, we express the horizontal derivative in terms of partial and fiber derivatives as
\begin{multline*}
\left\langle \frac{\delta l}{\delta\rho},\delta\rho\right\rangle=\left\langle \frac{\partial l}{\partial\rho},\delta\rho\right\rangle +\left\langle \frac{\delta l}{\delta\Xi},\mathcal{A}(0,\delta\rho,0,0)\times\Xi\right\rangle +\left\langle \frac{\delta l}{\delta\xi},\mathcal{A}(0,\delta\rho,0,0)\times\xi\right\rangle.
\end{multline*}
Therefore, the horizontal Lagrange--Poincar\'e equation of step two coincides with equation \eqref{LP1hor1}. In conclusion, equations \eqref{LP1vert}, \eqref{LP1hor1} and \eqref{LP1hor2} are obtained as well in the second step of reduction. However, in the first step of reduction equation \eqref{LP1hor2} is an horizontal equation whereas in step two it is vertical. Furthermore, the reconstruction equation from the second step to the first one is easily seen to be 
\begin{equation}
\partial_t \theta_s=\partial_s\theta_t.
\end{equation}

\begin{remark}The trivial Maurer-Cartan connection may not always be the more convenient for a particular Lagrangian. In Mechanics, the study of a rigid body with rotors in \cite{Scheurle} is an example where the appropriate connection adapted to the Lagrangian is the mechanical connection (the connection defined by a Riemannian metric). Future work will study this example using the techniques of reduction by stages developed in \cite{BC}.
\end{remark}

\subsection{A particular Lagrangian} We now choose the particular Lagrangian in $J^1P$; 
\begin{align*}
L(r,r_s,r_t,\Lambda,\Lambda_s,\Lambda_t,\theta,&\theta_s,\theta_t)=\frac{1}{2}\langle r_t,r_t\rangle+ \frac{1}{2}\langle \Lambda^{-1}\Lambda_t ,I\Lambda^{-1}\Lambda_t\rangle\\&+\frac{1}{2}\langle \Lambda^{-1}\Lambda_t+\theta_t,K(\Lambda^{-1}\Lambda_t+\theta_t)\rangle-E(\Lambda^{-1}\Lambda_s,\theta_s,\langle r,r\rangle),
\end{align*}
where $I$ is the inertia tensor of the configuration of the strand, $K$ is the inertia tensor of the rotors and $E$ is a function called potential energy. This Lagrangian combines terms appearing in the Lagrangian given for the rigid body with rotors in \cite{Scheurle} and the Lagrangian proposed in \cite{Mol Strand} for the molecular strand. It is invariant by $SO(3)\times\sS\times\sS\times\sS$. If the first step of reduction is performed by group $SO(3)$ with the Maurer-Cartan connection, that is $\tilde{\mathcal{A}}=0$, then the reduced Lagrangian is:
\begin{align*}
l(\rho,\rho_s,\rho_t,\theta,\theta_s,\theta_t,\Omega,\omega)=&\frac{1}{2}\langle \rho_t+\omega\times\rho,\rho_t+\omega\times\rho\rangle+ \frac{1}{2}\langle \omega,I\omega\rangle\\&+\frac{1}{2}\langle \omega+\theta_t,K(\omega+\theta_t)\rangle-E(\Omega,\theta_s,\langle \rho,\rho\rangle),
\end{align*}
For this Lagrangian the vertical Lagrange--Poincar\'e equations \eqref{LP1vert} takes the form
\begin{align*}
0=&\rho\times(\rho_{tt}+2\omega\times\rho_t+\omega_t\times\rho+\langle\omega,\rho\rangle\omega)+(I+K)\omega_t+K\theta_{tt}\\&+\omega\times((I+K)\omega+K\theta_t)-\partial_s\frac{\delta E}{\delta\Omega}+\Omega\times\frac{\delta E}{\delta\Omega},
\end{align*}
while the horizontal Lagrange--Poincar\'e equations \eqref{LP1hor1} and \eqref{LP1hor2} are
\begin{align*}
\omega\times(\rho\times\omega-2\rho_t)-\rho_{tt}-\omega_t\times\rho&=2\rho\frac{\partial E}{\partial \langle\rho,\rho\rangle},
\\
K\omega_t+K\theta_{tt}&=\partial_s\frac{\delta E}{\delta\theta_s},
\end{align*}
together with the reconstruction equation is 
$$\omega_s-\Omega_t=\Omega\times\omega.$$
In turn, the second step of reduction defines the reduced Lagrangian
\begin{align*}
l(\rho,\rho_s,\rho_t,a,b,\Omega,\omega)=&\frac{1}{2}\langle \rho_t+\omega\times\rho,\rho_t+\omega\times\rho\rangle+ \frac{1}{2}\langle \omega,I\omega\rangle\\&+\frac{1}{2}\langle \omega+b,K(\omega+b)\rangle-E(\Omega,a,\langle \rho,\rho\rangle),
\end{align*}
whose vertical Lagrange--Poincar\'e equations;
\begin{align*}
&0=\rho\times(\rho_{tt}+2\omega\times\rho_t+\omega_t\times\rho+\langle\omega,\rho\rangle\omega)+(I+K)\omega_t+Kb_{t}\\&\hspace{50mm}+\omega\times((I+K)\omega+Kb)-\partial_s\frac{\delta E}{\delta\Omega}+\Omega\times\frac{\delta E}{\delta\Omega},
\\
&K\omega_t+Kb_{t}=\partial_s\frac{\delta E}{\delta a},
\end{align*}
the horizontal Lagrange--Poincar\'e equations are;
\begin{align*}
\omega\times(\rho\times\omega-2\rho_t)-\rho_{tt}-\omega_t\times\rho&=2\rho\frac{\partial E}{\partial \langle\rho,\rho\rangle},
\end{align*}
and reconstruction equation $a_t=b_s$, which makes this equations equivalent to the ones in the previous step.

{\small \noindent\textbf{Authors' addresses:} }

{\small \smallskip}

{\small \noindent M.A.B.:
Facultad de Ciencias Matem\'aticas, Universidad Complutense de Madrid, Plaza de
Ciencias 3, 28040--Madrid, Spain.
\newline
\emph{E-mail:\/} \texttt{mberbel@mat.ucm.es}
}

{\small \smallskip}

{\small \noindent M.C.L.:
Facultad de Ciencias Matem\'aticas, Universidad Complutense de Madrid, Plaza de
Ciencias 3, 28040--Madrid, Spain.
\newline
\emph{E-mail:\/} \texttt{mcastri@mat.ucm.es}
}

\end{document}